\documentclass{amsart}
\usepackage{graphicx}
\usepackage{amsmath}
\usepackage{amssymb}
\usepackage[all]{xy}
\usepackage{hyperref}
\newcommand{\bfi}{\bfseries\itshape}
\newtheorem{theorem}{Theorem}
\theoremstyle{plain}

\newtheorem{definition}{Definition}

\newtheorem{lemma}{Lemma}

\newtheorem{proposition}{Proposition}
\newtheorem{remark}{Remark}
\numberwithin{equation}{section}

\newcommand{\ed}{{\mathbf{d}}}
\newcommand{\ad}{{\mathrm{ad}}}
\newcommand{\h}{{\mathfrak{h}}}
\newcommand{\A}{{\mathcal{A}}}
\newcommand{\tA}{{\tilde{\mathcal{A}}}}
\newcommand{\Orbit}{{\mathcal{O}}}

\newcommand{\g}{{\mathfrak{g}}}

\newcommand{\ddt}{{\left. \frac{d}{dt}\right |_{t=0}}}

\begin{document}

\title[On singular Poisson Sternberg spaces]
{On singular Poisson Sternberg spaces}
\author{M. Perlmutter}
\address{(MP) Institute of Fundamental Sciences, Massey University, Palmerston North, New Zealand}
\email{M.Perlmutter@massey.ac.nz}
\author{M. Rodriguez-Olmos}
\address{(MR-O) Section de Mathematiques, Ecole Polythecnique F\'ed\'erale de Lausanne, Switzerland}
\email{miguel.rodriguez@epfl.ch}

\begin{abstract}
We obtain a theory of  stratified Sternberg spaces thereby ex\-ten\-ding the theory of cotangent
bundle reduction for free actions to the singular case where the action on the base
manifold consists of only one orbit type. We find that the symplectic reduced spaces
are stratified topological fiber bundles over the cotangent bundle of the orbit space. We also obtain a Poisson stratification of the Sternberg space.
To construct the singular Poisson Sternberg space we develop an appropriate theory of singular connections
for proper group actions on a single orbit type manifold including a theory of holonomy extending the usual Ambrose-Singer theorem for principal bundles.

\end{abstract}

\maketitle

\section{Introduction}
In this paper we consider the problem of cotangent bundle reduction for a proper action of
a Lie group on a manifold with the simplifying assumption that the base manifold
on which the group acts consists of just one orbit type. This is a major simplifiying assumption to the more general problem where there are multiple orbit types on the base manifold. However, the resulting theory is already interesting and leads to a generalization of the theory of connections on a principal bundle. We will motivate the theory with a class of examples generated by homogeneous spaces where a group $G$ acts on $G/H$ and then on its cotangent bundle by the lifted action. Since this is a transitive action it is clear that there is just one orbit type. This example will appear later as a fundamental ingredient of the theory.

After reviewing some preliminary results on cotangent bundle reduction  in Section 2, we begin, in Section 3, with the transitive case of a base manifold that is a homogeneous space.  We consider the quotient of $T^{\ast}(G/H)$ by the action of the group $G$. Using commuting reduction by stages, we show that this is a Poisson stratified space with strata determined by the coadjoint induced action of $H$ on $\mathfrak{h}^\circ$. If we denote an element of the isotropy lattice for this action by $(K)$, then the Poisson strata are given by $\mathfrak{h}^\circ_{(K)}/H$. The Poisson bracket on each stratum is induced by the Lie-Poisson structure on $\mathfrak{g}^{\ast}$. We also show that the symplectic leaves of each stratum are given by $(\mathcal{O}_{\mu}\cap \mathfrak{h}^\circ_{(K)})/H$. Notice that this result is the singular generalization of the Lie-Poisson structures and coadjoint orbits, in the sense that if the action is free, then $H=e$ and Poisson and symplectic reduction produce the Lie-Poisson structure and the Kostant-Kirillov-Souriau (KKS) structure on the coadjoint orbits respectively.

Significantly, this particular case of singular cotangent bundle reduction does not require a connection precisely because the action on the base manifold is transitive so every tangent vector on the $G$-principal bundle $G/H\rightarrow \{\cdot\}$ is vertical and therefore every nonzero covector has nonzero momentum, i.e. the splitting of the cotangent bundle into zero momentum and nonzero momentum covectors is trivial.

Next we consider the non-transitive case. We need to first split the
tangent bundle of the base manifold $M$ into vertical and horizontal
distributions. This can be done with a $G$-invariant metric,
guaranteed by the properness of the action. The next step,
undertaken in Section 4, is to associate a connection to this
splitting. At first glance this can't be done in the usual way since
a surjective mapping from the tangent space at any point to the Lie
algebra will have kernel with dimension equal to the codimension of
the algebra. On the other hand, the horizontal spaces have
codimension equal to
$\mathrm{dim}\,\mathfrak{g}-\mathrm{dim}\,\mathfrak{g}_m$ where
$\mathfrak{g}_m$ is the nontrivial stabilizer algebra at the point
$m$. The resolution is to form a vector bundle $\nu\rightarrow M$,
thanks to the properness of the action, whose fibers are isomorphic
to $\mathfrak{g}/\mathfrak{g}_m$ and define a connection, which we
call a \emph{singular} connection, as a surjective bundle map
covering the identity from $TM$ to $\nu$. This leads to an invariant
splitting of the tangent bundle.

With this in place we study the orbit type stratification of $TM$ and prove, in Theorem \ref{lattice theorem}, that the isotropy lattice is determined by the action of $H$ on $\mathfrak{g}/\mathfrak{h}$ where $\mathfrak{h}$ is the stabilizer algebra at some point $m\in M$. Relative to the splitting induced from the connection, we explicitly determine the stratification of $TM$ and its quotient $(TM)/G$ and we also write down a stratified version of the Atiyah sequence for a principal bundle.

In 4.3 we introduce the curvature of the singular connection. Since the singular connection is an Ehresmann connection, we can use the curvature theory for an Ehresmann connection rather than attempt to define an exterior derivative of a bundle map. Using this as a starting point, we are able to prove, in Proposition \ref{proposition curvature}, that the curvature is a $G$-equivariant bundle map $\wedge^2TM\rightarrow \nu$ that takes values in the stratum of $\nu$ that contains the zero section.

In addition, in Theorem \ref{Ambrose Singer} we prove an Ambrose-Singer theorem for the singular connection which demonstrates that the holonomy group at a point $m\in M$ is contained as a subgroup in $N(H)/H$ where $H$ is the stabilizer of $m$. Along the way to doing this, we obtain a one-to-one correspondence between singular connections on $M\rightarrow M/G$ and principal connections on the bundle $M_{H}\rightarrow M_{H}/(N(H)/H)$. This concludes our study of the geometry of the singular connection.

In Section 5, we apply the singular connection to the construction
of a connection dependent realization of the symplectic structure on
the  reduced spaces $\mathbf{J}^{-1}(\mathcal{O})/G$, where
$\Orbit\in\g^*$ is a chosen coadjoint orbit contained in the image
of the momentum map $\mathbf{J}:T^{\ast}M\rightarrow \g^{\ast}$
associated to the cotangent lifted action of $G$ (given by
$\left\langle\mathbf{J}(\alpha_m),\xi\right\rangle=\left\langle
\alpha_m,\xi_{M}(m)\right\rangle$). The construction follows the one
for the Sternberg space when the action is free \cite{St, PeRa}.
This is the  content of Theorem \ref{minimal coupling theorem},
which shows that the singular reduced Sternberg space is a bundle
over $T^*(M/G)$ whose symplectic fibers are
$(\mathcal{O}_{\mu}\cap\mathfrak{h}_{(K)}^\circ)/H$ which are shown,
in Section 3, to be the symplectic leaves of the Poisson strata of
$\mathfrak{h}^\circ/H$.

Finally, in Section 6 we compute the full Poisson stratification of
$(T^{\ast}M)/G$ in the Sternberg representation. The strata of
$(T^{\ast}M)/G$ are determined by the $H$-isotropy lattice of
$\mathfrak{h}^\circ$. We show in Theorem \ref{gauge poisson theorem}
that, using the singular connection, we can realize each stratum as
a bundle over $T^{\ast}(M/G)$ with fibers isomorphic to the Poisson
strata in the homogeneous Lie-Poisson problem, that is
$\mathfrak{h}^\circ_{(K)}/H$. The Poisson  bracket obtained on this
space and it generalizes the gauge Poisson bracket in the free
theory \cite{Mont, PeRa}. The bracket consists of a canonical term
associated to the canonical symplectic structure on $T^{\ast}(M/G)$,
a coupling term that involves the reduced curvature of the singular
connection and a term involving the homogeneous Lie-Poisson
structure on the fibers $\mathfrak{h}^\circ_{(K)}/H$.

The theory developed for the problem with a single orbit type will play an important role in the
solution to the general problem of singular cotangent bundle reduction for base manifolds
admitting multiple orbit type which is the subject of a forthcoming paper \cite{PeRaRo}.

Finally we remark that a related approach to the problem studied in this paper
carried out in \cite{Ho}, following the alternative realization of
$(T^{\ast}M)/G$  due to Weinstein \cite{We} in the free case.

\noindent {\bf Acknowledgements.} M. Perlmutter wishes to
acknowledge the generous support of the Bernoulli Center where part
of the research for this paper was completed.

\section{Background and preliminaries}

\subsection{Proper actions with single orbit type}
Let $M$ be a single orbit type manifold with respect to the proper action of
the Lie group $G$. Thus $M=M_{(H)}$ (where $M_{(H)}:=\{m\in M : G_{m}=gHg^{-1} \mathrm{\ for\ some\ }g\in G\}$) for some compact subgroup $H\subset G$. It is well known (see \cite{DuiKol}) that the orbit
space $M_{(H)}/G$ is then a smooth manifold. One way to see this is to
consider the smooth submanifold $M_H$ of $M$ consisting of the points
in $M$ with stabilizer \emph{precisely} equal to $H$. It is easy to see that
the subgroup $N(H)$, the normalizer group of $H$ in $G$, acts on $M_H$ and that
every orbit in $M$ intersects $M_H$ on an $N(H)$ orbit. Furthermore, the quotient
group $N(H)/H$ acts freely on $M_H$ and generates the same orbit space.
Therefore we have
\begin{equation*}
M_{(H)}/G\simeq M_H/(N(H)/H)
\end{equation*}
and the right hand side is the base space of the principal $N(H)/H$-bundle,
$M_H\rightarrow M_H/(N(H)/H)$. We will see that the subgroup $N(H)$ plays
a crucial role in the geometry of connections that we are going to define.

\subsection{The Sternberg space for a free action}

Here we recall an important realization of the Poisson reduced space $(T^{\ast}Q)/G$ obtainable once a principal connection $\mathcal{A}$ on the principal bundle $\pi:Q\rightarrow Q/G$ is fixed. (To distinguish the free case from the singular one, we denote the manifold on which $G$ acts freely by $Q$ instead of $M$).
The connection allows us to realize the reduced space as a $\mathfrak{g}^{\ast}$ fiber bundle over the reduced cotangent bundle $T^{\ast}(Q/G)$. Detailed proofs of the results in this section are found in \cite{PeRa}.

The construction of the Sternberg space proceeds in two steps. First one pulls
back the configuration space bundle $\pi:Q\rightarrow Q/G$ by the cotangent
bundle projection $\tau_{Q/G}:T^{\ast}(Q/G)\rightarrow Q/G$ to obtain the
$G$-principal bundle
\begin{equation*}
\tilde{Q}=\{(\alpha_{[q]},q) \in T ^\ast (Q/G) \times Q : [q]=\pi(q), q \in Q \}
\end{equation*}
over $T^\ast(Q/G)$ with fiber over $\alpha_{[q]}$ diffeomorphic to
$\pi^{-1}([q])$. Recall that the $G $-action on $\tilde{Q} $ is given by $g
\cdot ( \alpha_{[q]}, q) : = ( \alpha_{[q]}, g \cdot q) $ for any $g \in G $
and $( \alpha_{[q]}, q) \in \tilde{Q} $. The diagram defining this pull back
bundle and the associated maps is given by

\unitlength=5mm
\begin{center}
\begin{picture}(10,6)
\put(1,5){\makebox(0,0){$\tilde{Q}$}}
\put(9,5){\makebox(0,0){$Q$}}
\put(1,0){\makebox(0,0){$T^\ast(Q/G)$}}
\put(9,0){\makebox(0,0){$Q/G$}}
\put(1,4){\vector(0,-1){3}}
\put(9,4){\vector(0,-1){3}}
\put(2.2,5){\vector(1,0){5.7}}
\put(2.6,0){\vector(1,0){5.4}}
\put(1,2){\makebox(0,0){$$}}
\put(9.5,2.5){\makebox(0,0){$\pi$}}
\put(0.4,2.5){\makebox(0,0){$\pi^{\sharp}$}}
\put(5,5.6){\makebox(0,0){$\tilde{\tau}$}}
\put(5,.7){\makebox(0,0){$\tau_{Q/G}$}}
\end{picture}
\end{center}
\bigskip
\medskip
where $\pi^{\sharp}$ and $\tilde{\tau}$ are the projections
onto the first and second factors respectively. The following fact about
$\tilde{Q}$ will be often used in the sequel.
\begin{proposition}
\label{isomorphism of q with annihilator of the vertical}
$\tilde{Q}$ is also a vector bundle over $Q$ isomorphic to the annihilator $V(Q)^\circ
\subset T^\ast Q$  of the vertical bundle $V(Q) \subset TQ$. The fibers of these vector
subbundles are given by $V(Q)_q := \ker T_q \pi = \{\xi_Q(q) : \xi \in \mathfrak{g} \}
\subset T_q Q$ and
$V(Q)^\circ_q : =  \{\alpha_{q}\in T^{\ast}_{q}Q : \left\langle
\alpha_{q},\xi_{Q}(q)\right\rangle =0\,\,\forall\xi\in\g\}
\subset T_{q}^\ast Q$ for each $q \in Q$. Consequently, $\tilde{Q}$ is bundle isomorphic to $\mathbf{J}^{-1}(0)$.
\end{proposition}

The second step is to form the coadjoint bundle of $\tilde{Q}$, that is, the associated vector
bundle to the $G$-principal bundle $\pi^\sharp:\tilde{Q} \rightarrow T^\ast (Q/G)$ given
by the coadjoint representation of $G$ on $\mathfrak{g}^\ast$.  The {\bfi
Sternberg space\/}, denoted by $S$, is thus defined by
\begin{equation*}
S:=\tilde{Q}\times_{G}\mathfrak{g}^{\ast}.
\end{equation*}
Abusing notation we will denote the fiber projection $\pi^{\sharp}:S\rightarrow T^{\ast}(Q/G)$ with the same symbol as the quotient map $\pi^{\sharp}:\tilde{Q}\rightarrow T^{\ast}(Q/G)$.
Using the connection $\mathcal{A}$ we then construct the bundle isomorphism
to the Poisson reduced space $(T^{\ast}Q)/G$ as follows.

\begin{proposition}
\label{prop_StoReduced}
The map $\varphi_{ \mathcal{A}}: \tilde{Q} \times \mathfrak{g}^\ast
\rightarrow T ^\ast Q $ given by
\begin{equation*}
\varphi_{\mathcal{A}}\left(\left(\alpha_{[q]}, q \right),
\mu\right): = T^{\ast}_{q}\pi(\alpha_{[q]})
+\mathcal{A}(q)^{\ast}\mu
\end{equation*}
is a $G$-equivariant vector bundle isomorphism over $Q$. It descends to a
vector bundle isomorphism over $Q/G$
\begin{equation*}
%\label{eq: definition of phi a}
\Phi_{\mathcal{A}}:S\rightarrow (T^{\ast}Q)/G.
\end{equation*}
\end{proposition}

The gauge Poisson bracket on $S$ is the pull back by $\Phi_{\mathcal{A}}$ of the the natural reduced Poisson
structure on $(T^{\ast}Q)/G$. In order to study it we first introduce the necessary notions of horizontal lifts and covariant derivatives in the context needed for our purposes.
First one constructs the horizontal lift on the $G$-bundle $\tilde{Q}\rightarrow T^{\ast}(Q/G)$ endowed with the connection form $\tilde{\mathcal{A}}:=\tilde{\tau}^{\ast}\mathcal{A}$.
Given a curve $\alpha_{[q]}(t)$ in $T^{\ast}(Q/G)$, one has $q_{h}(t)$, the horizontal lift at $q$
of the curve $[q](t)=\tau_{Q/G}(\alpha_{[q]}(t))$ relative to $\mathcal{A}$.
Then the curve $(\alpha_{[q]}(t),q_{h}(t))$ lies in $\tilde{Q}$, is horizontal (relative to $\tilde{A}$) and covers $\alpha_{[q]}(t)$. Denoting horizontal lift operators by $\mathrm{hor}$,
it follows that
\begin{equation*}
%\label{lift to Q tilde}
\operatorname{hor}_{(\alpha_{[q]},q)}\left(v_{\alpha_{[q]}}\right) = \left( v_{\alpha_{[q]}}, \operatorname{hor}_q\left(T_{\alpha_{[q]}}\tau_{Q/G} (v_{\alpha_{[q]}})\right) \right)\in T_{(\alpha_{[q]},q)}\tilde{Q}.
\end{equation*}

Now, $S$ is an associated bundle to $\tilde{Q}$, therefore, for $s=[(\alpha_{[q]},q),\mu]$,

\begin{equation*}
%\label{lift to S}
\mathrm{hor}_{s}(v_{\alpha_{[q]}})
=T_{((\alpha_{[q]},q),\mu)}\rho \left(\operatorname{hor}_{(\alpha_{[q]}, q)} v_{\alpha_{[q]}}, 0 \right)\in T_{s}S,
\end{equation*}
where $\rho:\tilde{Q}\times\g^{\ast}\rightarrow S$ is the orbit projection.
Finally, for $f\in C^{\infty}(S)$ and $s\in S$ define $\mathbf{d}_{\tilde{\mathcal{A}}}^S f(s)\in T^{\ast}_{\pi^{\sharp}(s)}T^{\ast}(Q/G)$ by
\begin{equation}
\label{covariant derivative}
\mathbf{d}_{\tilde{\mathcal{A}}}^S f(s)\left(v_{\alpha_{[q]}}\right):=\mathbf{d}f(s)\left(\mathrm{hor}_{s}\left(v_{\alpha_{[q]}}\right) \right).
\end{equation}

Denote the curvature of the connection $\mathcal{A}$ by $\rm{Curv}_{\mathcal{A}}$. The reduced curvature form is a bundle map from $\wedge^2(T(Q/G))$ to the adjoint bundle, $\tilde{\mathfrak{g}}:=Q\times_{G}\mathfrak{g}$. Recall that the adjoint bundle
$\widetilde{\mathfrak{g}}$ is defined as the quotient $\widetilde{\mathfrak{g}}: = (Q
\times \mathfrak{g})/G$ relative to the diagonal left $G $-action $g \cdot (q,
\xi) := (g \cdot q, \operatorname{Ad}_g \xi)$ on $Q \times \mathfrak{g}$,
where $g \in G$,
$q \in Q$, $\xi \in \mathfrak{g}$, and $\operatorname{Ad}_g$ is the adjoint
representation of $G$ on $\mathfrak{g}$. The adjoint bundle is a Lie algebra bundle
with base $Q/G$, that is, each fiber has a Lie algebra bracket depending smoothly on
the base.
The reduced curvature is then defined by
\begin{equation}
\label{eq:reducedcurv}
\mathcal{B}([q])(u_{[q]},v_{[q]})=[q,{\rm{Curv}_{\mathcal{A}}}(u_q,v_q)],
\end{equation}
where $u_q, v_q \in T_q Q$ are arbitrary vectors satisfying $T_{q}\pi(
u_{q})=u_{[q]}$,
$T_{q}\pi( v_{q})=v_{[q]}$ respectively and $[q,\xi]\in \tilde{\mathfrak{g}}$ denotes the $G$-class through $(q,\xi)$.

The (reduced) gauge Poisson bracket is then given by the following result.
\begin{theorem}
\label{thm:Global S Bracket}
Let $s=[(\alpha_{[q]},q),\mu]\in S$ and $v=[q,\mu] \in
\widetilde{\mathfrak{g}} ^\ast$.
The Poisson bracket of $f, g \in C^{\infty}(S)$ is given by
\begin{align*}
%\label{eq:Global S Bracket}
\{f,g\}_S(s)&=\Omega_{Q/G}(\alpha_{[q]})
\left(\mathbf{d}_{\tilde{\mathcal{A}}}^S
f(s)^{\sharp},\mathbf{d}_{\tilde{\mathcal{A}}}^S g(s)^{\sharp}\right)  \\
\nonumber & \qquad +\left\langle
v,\tilde{\mathcal{B}}(\alpha_{[q]})\left(\mathbf{d}_{\tilde{\mathcal{A}}}^S
f(s)^{\sharp},\mathbf{d}_{\tilde{\mathcal{A}}}^S
g(s)^{\sharp}\right)\right\rangle  -
\left\langle \mu,\left[\frac{\delta f}{\delta s},\frac{\delta g}{\delta s}\right]
\right\rangle,
\end{align*}
where $\Omega_{Q/G} $ is the canonical symplectic form on $T ^\ast(Q/G)$,
$\tilde{\mathcal{B}}\in \Omega^{2}(T^{\ast}(Q/G);\widetilde{\mathfrak{g}})$
is the $\widetilde{\mathfrak{g}}$-valued two-form on $T^{\ast}(Q/G)$
given by
\begin{equation*}
%\label{definition of tilde cal b}
\tilde{\mathcal{B}}=\tau_{Q/G}^{\ast}\mathcal{B},
\end{equation*}
with $\mathcal{B} \in \Omega^2(Q/G, \tilde{\mathfrak{g}})$ defined
in \eqref{eq:reducedcurv},
$\sharp:T^{\ast}(T^{\ast}(Q/G))\rightarrow T(T^{\ast}(Q/G))$ is the
vector bundle isomorphism induced by $\Omega_{Q/G}$,
 and  $\delta f/\delta s \in S ^\ast =
\tilde{Q} \times _G \mathfrak{g}$ is the usual fiber derivative of $f$ at the
point $s
\in S$, that is,
\begin{equation*}
%\label{fiber derivative in S}
\left\langle s', \frac{\delta f}{\delta s} \right\rangle : = {\left. \frac{d}{dt}\right |_{t=0}}f\left([(\alpha_{[q]},q),\mu+t\nu]\right)
\end{equation*}
for any $s' := [(\alpha_{[q]}, q), \nu)] \in S $.
\end{theorem}

\subsection{Minimal coupling forms}
The symplectic leaves of the Sternberg space are given by the submanifolds
in $S$ of the form $\tilde{Q}\times_G\mathcal{O}$, where $\mathcal{O}\subset \g^{\ast}$ is the coadjoint orbit through $\mu$. To describe the symplectic forms
on these spaces we need to recall the minimal coupling form due to Sternberg \cite{St}
which is a functorial construction of a presymplectic manifold associated to a principal
bundle with a connection and a Hamiltonian $G$-space.

Let $\sigma:Z \rightarrow B $ be a left principal $G$-bundle over the symplectic manifold $(B,\Omega)$, $\mathcal{L} \in
\Omega^1(Z; \mathfrak{g})$ a connection one-form on $Z$, $(F, \omega)$ a
Hamiltonian $G$-space with equivariant momentum map $\phi: F \rightarrow
\mathfrak{g}^\ast$, and denote by $\Pi_{Z}: \tilde Z \times F \rightarrow Z $ and
$\Pi_F: Z \times F \rightarrow F $ the two projections.

It can be shown that the closed two-form
 $\mathbf{d}\left\langle \Pi_F ^\ast \phi,  \Pi_{Z} ^\ast
\mathcal{L} \right\rangle +  \Pi_F ^\ast \omega$ descends to a closed two form
$\omega^\mathcal{L} \in \Omega^2(Z \times _G F) $, that is, $\omega^\mathcal{L}
$ is characterized by the relation
\[
\rho^\ast \omega^\mathcal{L} = \mathbf{d}\left\langle \Pi_F ^\ast
\phi,  \Pi_{Z} ^\ast \mathcal{L} \right\rangle +  \Pi_F ^\ast \omega,
\]
where $\rho: {Z} \times F \rightarrow {Z} \times_G F$ is the projection to the
orbit space.

Now  denote by $\sigma_F: {Z} \times _G F \rightarrow B$ the associated
fiber bundle projection given by $\sigma_F([z, f]) : = \sigma(z)$. Then
$\sigma_F ^\ast \Omega $ is also a closed two-form on $Z \times_G F $ and one
gets the {\bfi minimal coupling\/} presymplectic form $\omega^\mathcal{L} +
\sigma_F ^\ast \Omega $. In general, this presymplectic form is degenerate, but in the
crucial case below it is in fact a reduced symplectic form.

\subsection{Symplectic leaves of the Sternberg Space}
Let us apply the previous construction to the situation  $Z=\tilde Q,\, B =
T ^\ast Q $, $\Omega = \Omega_{Q/G} = - \mathbf{d} \Theta_{Q/G} $ (the
canonical symplectic  form on the cotangent bundle $T ^\ast(Q/G) $),
$\sigma = \pi^{\sharp}: ( \alpha_{[q]}, q) \in \tilde{Q} \mapsto \alpha_{[q]}
\in  T ^\ast Q$, and $\mathcal{L}=\tilde{\mathcal{A}} $ as in the previous paragraph. Choose also $(F ,
\omega) = ( \mathcal{O}, \omega_ \mathcal{O}^-)$, where $\omega_{\mathcal{O}^-}$ is the (-) KKS form on $\mathcal{O}$ and
$\phi =
\mathbf{J}_ \mathcal{O}: \mathcal{O} \rightarrow \mathfrak{g}^\ast$ given by
$\mathbf{J}_ \mathcal{O}( \mu) = - \mu$ for any $\mu\in \mathcal{O}$ its associated $G$-equivariant momentum map for the coadjoint representation.

Then
$\pi^{\sharp}:\tilde{Q} \times _G \mathcal{O}
\rightarrow T^\ast(Q/G) $ is given by,
$\pi^{\sharp}([( \alpha_{[q]},q), \mu]) = \alpha_{[q]}$.
Denote the two form $\omega^ \mathcal{\tilde A}$ in this situation by $\tilde{\omega}_
\mathcal{O}^- $ and hence it is uniquely characterized by the relation
\begin{equation}
\label{tilde omega minus}
\rho^\ast \tilde{\omega}_\mathcal{O}^- = \mathbf{d}\left\langle\Pi_
\mathcal{O}^\ast \mathbf{J}_ \mathcal{O}, \Pi_{\tilde{Q}} ^\ast \tilde{
\mathcal{A}} \right\rangle + \Pi_ \mathcal{O} ^\ast \omega_ \mathcal{O}^- .
\end{equation}

The {\bfi minimal coupling} form in this situation is

\begin{equation}
\label{definition minimal coupling form}
\omega_{\mathrm{min}}^{\mathcal{O}}:=\tilde{\omega}_{\mathcal{O}}^{-}+(\pi^{\sharp})^{\ast}\Omega_{Q/G}.
\end{equation}

We then have the following theorem that says that the minimal coupling
form coincides with the reduced symplectic form on the leaves of the Sternberg spaces.

\begin{theorem}
\label{theorem: coupling form reduction}
The symplectic leaves of the Sternberg space $(S,\{,\}_{S})$ are given by $(\tilde{Q}\times_G\mathcal{O},\omega_{\mathrm{min}}^{\mathcal{O}})$ where $\mathcal{O}$ is a coadjoint orbit of $\mathfrak{g}^{\ast}$.
The minimal coupling two-form $\omega_{\mathrm{min}}^{\mathcal{O}}$ is the reduced symplectic form on the leaf obtained by orbit reduction, i.e. $\omega_{\mathrm{min}}^{\mathcal{O}}$ given in equation \eqref{definition minimal coupling form} is the unique two-form on $\tilde{Q}\times_{F}\mathcal{O}$ that satisfies
\begin{equation*}
%\label{orbit reduced form}
\rho^{\ast}\omega_{\mathrm{min}}^{\mathcal{O}}=\iota_{\mathcal{O}}^{\ast}\Omega-\mathbf{J}_{\mathcal{O}}^{\ast}\omega_{\mathcal{O}}^{+},
\end{equation*}
where $\Omega=\varphi_{\mathcal{A}}^{\ast}\omega_{Q}$ and $\iota_{\mathcal{O}}$ is the inclusion of $\tilde{Q}\times \mathcal{O}$ into $\tilde{Q}\times \mathfrak{g}^{\ast}$.

\end{theorem}

\section{Homogeneous Lie-Poisson Structures}
In this section we consider our most important example of a single
orbit type manifold, a homogeneous space. Here we compute the
reduced spaces and their stratifications ``by hand''. We will
later see that the reduced symplectic and Poisson stratified
structures on homogeneous  spaces appear as the symplectic and
Poisson fibers of the Sternberg spaces developed in Sections 5 and
6.

Let $\g^*$ be equipped with the $(-)$ Lie-Poisson structure $\{\cdot,\cdot\}^{-}$
and
consider the algebra of smooth $H$-invariant functions in $\g^*$,
$C^H(\g^*)$.
Since the Lie-Poisson bracket is $H$-invariant,
$(C^H(\g^*),\{\cdot,\cdot\}^{-})$ is
a Poisson algebra. Consider also the action of $H$ on $\h^\circ$ induced
by the coadjoint action.  Denote by $(K)$ an element of the isotropy lattice for this action. We then have the following,
\begin{lemma}
\label{invariance lemma}
The Hamiltonian vector fields of functions in $C^H(\g^*)$
leave
invariant $\h^\circ$ and its orbit types $\h^\circ_{(K)}$.
\end{lemma}
\begin{proof}
Let $f\in C^H(\g^*)$. Then its Hamiltonian vector field evaluated at
$\mu\in\h^\circ$ is
$\ad^*_{\frac{\delta f}{\delta\mu}}\mu$. Let $\lambda\in\h$. Then
$$\langle \ad^*_{\frac{\delta
f}{\delta\mu}}\mu,\lambda\rangle=\langle-\ad^*_\lambda\mu,\frac{\delta
f}{\delta\mu}\rangle=-\ed f(\mu)\cdot (\ad^*_\lambda\mu)=0,$$
since $f$ is $H$-invariant. Thus $\ad^*_{\frac{\delta
f}{\delta\mu}}\mu\in\h^\circ=T_\mu\h^\circ$, so $\h^\circ$ is left invariant by
the
Hamiltonian flow of
$f$. Since this flow is necessarily $H$-equivariant, each of its orbits
consists
of points with the same isotropy, so the orbit types of $\h^\circ$ are
also
preserved.
\end{proof}
By the previous lemma and Lemma 2 of \cite{LeBa},
the $H$-invariant functions vanishing
on $\h^\circ$ and $\h_{(K)}^\circ$ are Poisson ideals of $C^H(\g^*)$,
denoted by
$I(\h^\circ)$ and $I(\h^\circ_{(K)})$ respectively.
Hence $C^H(\h^\circ)=C^H(\g^*)/I(\h^\circ)$ is a reduced Poisson algebra
and
$C^H(\h^\circ_{(K)})=C^H(\g^*)/I(\h^\circ_{(K)})$ a reduced Poisson
subalgebra.
Since the algebra
of smooth functions of $\h^\circ/H$ as a singular reduced space is
defined precisely as
$C^H(\h^\circ)$, we have that there is a reduced Poisson bracket on
$C^{\infty}(\h^\circ/H)$
given by
\begin{equation}
\label{homogeneous bracket}
\{f,g\}([\mu])=-\left\langle\mu,\left[\frac{\delta F}{\delta
\mu},\frac{\delta
G}{\delta\mu}\right]\right\rangle,
\end{equation}
where $[\mu]\in \h^\circ/H$, $f,g\in C^{\infty}(\h^\circ/H)=C^H(\h^\circ)$ and
$F,G$ are
smooth $H$-invariant extensions to $C^H(\g^*)$ of $f$ and $g$
respectively.
Then,
in view of Lemma \ref{invariance lemma}, and since for any orbit type $\h^\circ_{(K)}/H$
is smooth,
this reduced Poisson algebra restricts to a reduced smooth Poisson
structure on
the smooth
stratum $\h^\circ_{(K)}/H$ of $\h^\circ/H$, making it a Poisson
manifold.

In the case of a $G$-homogeneous space $M$ we can identify  $M=G/H$ where
$H$ is a compact isotropy group and we take the quotient to be for the right action of $H$ on $G$. Then $G$ acts on the left on $G/H$ according to
$g\cdot [g']=[gg']$. It is clear that the stabilizer group of the point $[g]$ is then $gHg^{-1}$
 and therefore $M=M_{(H)}$.

We next consider the problem of singular symplectic reduction for the cotangent lifted action $G\times T^{\ast}(G/H)\rightarrow T^{\ast}(G/H)$. The bulk of this paper is devoted to obtaining
gauge realizations of the singular reduced Poisson and symplectic spaces for more general single orbit type base manifolds.
 We will see that  this particular example of a homogeneous
space will appear and play a role analogous to that of a coadjoint orbit in the free case. In this sense
the symplectic reduction of the homogeneous space is precisely the correct generalization in the singular setting of a coadjoint orbit in the regular case which, recall, is obtained by the regular symplectic
reduction for the action $G\times T^{\ast}G\rightarrow T^{\ast}G$, which is the cotangent lift of the left translation of $G$ on itself.

We can carry out the reduction using the technique of commuting reduction by stages (see for example
\cite{MaMiOrPeRa}).
We consider the left action of $G\times H$ on $T^{\ast}G$
which is the cotangent lift of the action on $G$ given by
\begin{equation*}
(g,h)\cdot g'=gg'h^{-1}.
\end{equation*}
It is then clear that the restricted
actions of the two
subgroups of $G\times H$, $G$, and $H$ commute. While the total action is not free,
the restricted actions are free actions. Consider reduction at the momentum value
$(\mu,0)\in \mathfrak{g}^{\ast}\times \mathfrak{h}^{\ast}$. We first reduce by the $H$-action at zero momentum to obtain $T^{\ast}(G/H)$, equipped with the canonical symplectic form, by the regular cotangent bundle reduction theorem at zero momentum (see \cite{AbMa}). The remaining action is then given by $G \times T^{\ast}(G/H)\rightarrow T^{\ast}(G/H)$ precisely as in the starting point for the original singular reduction problem.

We next reverse the order of reduction which we are allowed to do by the singular commuting reduction theorem (\cite{SjLe}).
Reducing at $\mu\in\mathfrak{g}^{\ast}$ the left $G$-action on $T^{\ast}G$ we obtain
$\mathcal{O}$, the coadjoint orbit through $\mu$, with the usual ($-$) KKS symplectic form. We then consider the action
$H\times \mathcal{O}\rightarrow \mathcal{O}$.
As this action is the restriction of the
coadjoint action of $G$ on $\mathcal{O}$ it is a Hamiltonian action with momentum map
given by $\mathbf{J}_{H}(\mu')=-\iota^{\ast}\mu'\in \mathfrak{h}^{\ast}$
where $\iota:\mathfrak{h}\hookrightarrow \mathfrak{g}$ is the canonical inclusion and $\mu'\in \mathcal{O}$. From this expression it follows that
\begin{equation*}
\mathbf{J}_{H}^{-1}(0)=\mathcal{O}\cap \mathfrak{h}^\circ.
\end{equation*}
 We know by the Sjamaar-Lerman theory \cite{SjLe} of symplectic stratifications that the symplectic quotient $\mathbf{J}_{H}^{-1}(0)/H$ is a stratified topological space with symplectic strata given by $(\mathbf{J}_{H}^{-1}(0)_{(K)})/H$ with $(K)\leq(H)$ where the
$K$ are the stabilizer subgroups determined from the action $H\times \mathfrak{h}^\circ\rightarrow \mathfrak{h}^\circ$ (which is just the coadjoint action restricted to the subgroup $H$) and $(K)\leq (K')$ means that any representative in the conjugacy class  $(K)$ is conjugate to a subgroup of a representative in $(K')$.
Therefore we obtain that the symplectic reduced spaces, for each $\mu\in \mathfrak{g}^{\ast}$, for the action of $G$ on $T^{\ast}(G/H)$ have symplectic strata given by
\begin{equation*}
(\mathcal{O}\cap \mathfrak{h}^\circ_{(K)})/H  \qquad (K)\leq (H).
\end{equation*}
Notice that if $\mu$ is such that  $\Orbit$, the coadjoint orbit containing $\mu$, satisfies $\mathcal{O}\cap \mathfrak{h}^\circ=\varnothing$ then  $\mathbf{J}^{-1}(\mu)=\varnothing$ where $\mathbf{J}$ denotes the momentum map for
the action $G\times T^{\ast}(G/H)\rightarrow T^{\ast}(G/H)$. In other words the image of this
momentum map takes values $\mu$  that satisfy $\mathcal{O}\cap \mathfrak{h}^\circ\neq \varnothing$.

\begin{remark}
The symplectic structure on the strata, $(\mathcal{O}\cap \h^\circ_{(K)})/H$ is given by the quotient of the restriction of the symplectic form of
the coadjoint orbit. Therefore, denoting this structure by $\omega_{(K)}$ one has the following formula
\begin{equation}
\label{stratifiedKK}
\pi_{(K)}^{\ast}\omega_{(K)}=\iota_{(K)}^{\ast}\omega_{\mathcal{O}}^{-}
\end{equation}
where $\pi_{(K)}:\mathcal{O}\cap\mathfrak{h}^\circ_{(K)}\rightarrow (\mathcal{O}\cap\mathfrak{h}^\circ_{(K)})/H$,
$\iota_{(K)}:\mathcal{O}\cap\mathfrak{h}^\circ_{(K)}\hookrightarrow \mathcal{O}$
and $\omega_{\mathcal{O}}^-$ is the $(-)$ KKS symplectic structure on $\mathcal{O}$. From our construction
it follows that these symplectic strata are also the symplectic leaves of $\h^\circ_{(K)}/H$ for the smooth Poisson structure induced by the Poisson structure \eqref{homogeneous bracket} in $\h^\circ/H$. We will refer to the structures defined by \eqref{homogeneous bracket} and \eqref{stratifiedKK} by homogeneous Lie-Poisson bracket and homogeneous KKS form, repectively.
\end{remark}

\section{Singular Connections}
We wish to extend the concept of a principal connection to the
setting of a single orbit type manifold. In the rest of the paper we
will have $M=M_{(H)}$. Let $G\times M\rightarrow M$ be a proper
action. Then each $m\in M$ has isotropy group $G_{m}$ conjugate to
$H$ in $G$ with Lie algebra $\mathfrak{g}_{m}$. We first need to
generalize the target space for a connection in this singular
setting as we no longer have a fixed Lie algebra, but rather a
family of spaces $\mathfrak{g}/\mathfrak{g}_m$ for each $m\in M$
onto which the connection must project.

We denote by  $\mathfrak{m}:=\bigcup_{m\in M}\mathfrak{g}_{m}$. Because we assume the action is proper, we can prove that this space is a vector bundle as follows.
\begin{proposition}
The set $\mathfrak{m}$ is a smooth vector bundle over $M$. We call it the stabilizer
bundle over $M$.
\end{proposition}
\begin{proof}
This is a simple application of the tube theorem. We must show that
$\mathfrak{m}$ is locally
trivializable. Fix $m\in M$. Without loss of generality we can assume that
$G_m=H$. Let
$S_m=T_{m}(G\cdot m)^{\perp}$ with respect to some $G$-invariant metric
on $M$ (available by the properness of the action). Then $S_m$
is then a linear slice for the $G$-action at $m$ and $\exp_m$ is an
$H$-equivariant diffeomorphism
from an open ball $B$ containing the origin in $S_m$ to an $H$-invariant
submanifold transverse
to $G\cdot m$ at $m$. We have that $\phi: G\times_{H}B\rightarrow U$ given by
 \begin{equation*}
\phi([g,v])= g\cdot \exp_m v.
\end{equation*}
is a $G$-equivariant local diffeomorphism onto a
$G$-invariant neighborhood of the
orbit $G\cdot m$. Note that $\phi([e,0])=x$. Choose now a $H$-invariant Riemannian metric on $G$
and split $\g=T_eG$ as
$\g=\h\oplus\mathfrak{k}$. Then $\mathfrak{k}:=\h^{\perp}$ is a slice at the identity for
the free $H$-action on $G$, and therefore we can use again the tube theorem
to construct a local $H$-diffeomorphism $H\times O\rightarrow G$ onto a
neighborhood of $H$ in $G$. Here $O$ is a sufficiently small open ball around $0$ in
$\mathfrak{k}$.
This diffeomorphism is explicitly given by $(h,k)\mapsto h\cdot\exp_e k$,
where $\exp_e$ is the associated Riemannian exponential on $G$, for which $O$ lies inside its
domain of injectivity. This diffeomorphism drops to a
diffeomorphism $(H\times O)/H=O\rightarrow O'\subset G/H$ where $O'$ is a neighborhood in $G/H$ containing $[e]$.
We can therefore identify each element $[g]$
of $O'$
with $\exp_e k$ for a unique $k\in O$. Shrink $O'$ if necessary so
that it becomes a trivializing
neighborhood for the associated bundle $G\times_H B$ over $O'$ and call the
induced trivial bundle chart $\Psi:O'\times B\rightarrow U\subset M$. This map is given by $\Psi(\exp_{e}k,v)= \exp_{e}k\cdot \exp_{m}v$. Thus $U\subset M$ is a (not  invariant)
neighborhood of $x$. Finally, we can construct a trivialization of
$\mathfrak{m}$ over $U$, as follows: $f:U\times \h\rightarrow \mathfrak{m}_U$
is
given by
$$f(\exp_e k\cdot \exp_{m}s,\xi):=(\exp_ek\cdot\exp_m s,\mathrm{Ad}_{\exp_ek}\xi).$$
\end{proof}

Next, we form the bundle that will play the role of the Lie algebra in the standard
theory of connections on principal bundles. The fibers of this bundle, over each point
in $M$, should be isomorphic to the tangent space of the group orbit at that point.
The natural candidate for this fiber, over $m\in M$ is simply $\mathfrak{g}/\mathfrak{g}_m$. Consider the trivial bundle over $M$, $M\times \mathfrak{g}$. We then have the following injective
inclusion of vector bundles over $M$ covering the identity on the base
\begin{equation*}
\mathfrak{m}\hookrightarrow M\times\mathfrak{g}.
\end{equation*}
\begin{definition}
Let $\nu$ be the quotient bundle defined by,
\begin{equation*}
\nu:=\frac{M\times \mathfrak{g}}{\mathfrak{m}}.
\end{equation*}
\end{definition}
Note that the fiber over $m\in M$ is simply
\begin{equation*}
\nu_m=\mathfrak{g}/\mathfrak{g}_m.
\end{equation*}

We then have the following properties of the vector bundle $\nu$.

\begin{proposition}
\label{properties of nu}
The vector bundle $\nu$ satisfies:
\begin{enumerate}
\item $\nu|_{M_H}=M_H\times \mathfrak{g}/\mathfrak{h}$, i.e. the restriction of $\nu$ to
the constant stabilizer manifold $M_H$ is a trivial bundle.
\item There is a smooth action of the group $G$ on $\nu$ which is linear
on the fibers and covers the $G$-action on $M$. This action is defined by
\begin{equation}
\label{equation nu action}
g\cdot [\xi]_m=[\operatorname{Ad}_{g}\xi]_{g\cdot m}.
\end{equation}
\item With respect to this action, $\nu$ is a saturated vector bundle, i.e.
\begin{equation*}
\nu=G\cdot \nu|_{M_H}.
\end{equation*}

\end{enumerate}

\end{proposition}
\begin{proof}
For $(1)$, because the isotropy algebra for any point $m\in M_{H}$ is just $\mathfrak{h}$, the stabilizer bundle $\mathfrak{m}$ restricted to $M_{H}$ is the trivial bundle,
$
\mathfrak{m}|_{M_{H}}=M_{H}\times \mathfrak{h},
$
and therefore $$\nu|_{M_{H}}=\frac{M_{H}\times \mathfrak{g}}{\mathfrak{m}|_{M_H}}=M_{H}\times \mathfrak{g}/\mathfrak{h}.$$
For $(2)$, it is clear that if equation \eqref{equation nu action} is well defined, then it is an action.
To see that it is well defined, choose another representative $\xi'$ of the class $[\xi]_{m}$, so that
$\xi'=\xi+\lambda$ where $\lambda\in \mathfrak{g}_{m}$. Then $\operatorname{Ad}_{g}\xi'=\operatorname{Ad}_{g}\xi+\operatorname{Ad}_{g}\lambda$. But,
$\operatorname{Ad}_{g}\lambda\in \mathfrak{g}_{g\cdot m}$ and therefore $[\operatorname{Ad}_{g}\xi']_{g\cdot m}=[\operatorname{Ad}_{g}\xi]_{g\cdot m}$ as required.
To prove $(3)$, it is clear that since $M=G\cdot M_{H}$, any point $m'\in M$ can be written as
$m'=g\cdot m$ for some $m\in M_{H}$ and therefore since $\mathfrak{g}_{m'}=\operatorname{Ad}_{g}\mathfrak{h}$ it follows that $g\cdot \nu_{m}=\nu_{m'}$.
\end{proof}

\begin{remark}
When the action is free, $\mathfrak{h}=0$ and therefore the stabilizer bundle $\mathfrak{m}$ is $0$ so that $\nu=M\times \mathfrak{g}$ and therefore its quotient $\nu/G$ is the adjoint bundle $\tilde{\mathfrak{g}}=\frac{M\times \mathfrak{g}}{G}$.
\end{remark}

We next define the main object of this section, a singular connection for single orbit type manifolds.

\begin{definition}
A singular connection $\mathcal{A}$ for the single orbit type manifold $M$
is a smooth surjective vector bundle map $\mathcal{A}:TM\rightarrow \nu$, covering the identity,
with the properties:
\begin{itemize}
\item[(i)] $\mathcal{A}$ is $G$-equivariant: $\mathcal{A}(g\cdot v_{m})=g\cdot\mathcal{A}(v_m)$ for any $v_m\in TM$.

\item[(ii)] For all $\xi\in \mathfrak{g}$, $\mathcal{A}(\xi_M(m))=[\xi]_m.$

\end{itemize}
\end{definition}
Consequently for each $m\in M$, $\mathrm{ker}\mathcal{A}(m)$ is a complement to $\mathfrak{g}\cdot m\simeq \nu_m$ in $T_mM$ and together forms a $G$-invariant subbundle $H(M)$ of $TM$. Such connections always exist with our assumptions of a proper action, since it is well known (see \cite{DuiKol}) that, assuming the action is proper, there exists a $G$-invariant Riemannian metric on $M$. Using this metric, we can simply declare the horizontal space at a point $m\in M$ to be $H_m:=(\g\cdot m)^{\perp}$. It is clear that these spaces form a subbundle of $TM$ invariant under the $G$-action, and satisfy $TM=H(M)\oplus  V(M)\simeq H(M)\oplus \nu$.

\subsection{Stratification of $TM$}
In this section we use the singular connection to determine the orbit type stratification for the
tangent lifted action $G\times TM\rightarrow TM$. We will obtain an analogous result
when we dualize the action for the cotangent bundle. In studying the strata of the tangent
lifted action we can use the connection to reduce the problem to studying the
strata for the $G$-action on $\nu$. We obtain the following result.

\begin{theorem}
\label{lattice theorem}
The isotropy lattice for the action of $G$ on $TM$ is in one-to-one correspondence
with the lattice determined by the $\operatorname{Ad}$-induced action
\begin{equation*}
H\times \mathfrak{g}/\mathfrak{h}\rightarrow \mathfrak{g}/\mathfrak{h}.
\end{equation*}
Let $\mathcal{A}$ be a singular connection on $M$. Then we have a connection dependent $G$-equivariant diffeomorphism $\varphi_\mathcal{A}:TM\rightarrow H(M)\times\nu$ such that,  for each $(K)\leq (H)$ in the previous isotropy lattice it restricts to an equivariant diffeomorphism
\begin{equation*}
\varphi_\mathcal{A}|_{(TM)_{(K)}}:(TM)_{(K)}\rightarrow H(M)\times \nu_{(K)},
\end{equation*}
where $H(M)$ is the horizontal subbundle in $TM$ determined by $\mathrm{Ker}\mathcal{A}$ and $(K)$ refers to the conjugacy class of $K$ in $G$.

Since $\mathcal{A}$ is $G$-equivariant, then $\mathcal{A}:TM\rightarrow\nu$  is  a stratified morphism respecting the orbit type strata of $TM$ and $\nu$.
Also,
$\nu_{(K)}/G$ is a smooth fiber bundle over $M/G$ with typical fiber isomorphic to $(\mathfrak{g/h})_{(K)}/H$.
Finally,
$(TM)/G$ has the structure of a stratified vector bundle over $M/G$ with smooth strata $(TM)_{(K)}/G$ isomorphic to
\begin{equation*}
T(M/G)\times \nu_{(K)}/G.
\end{equation*}
\end{theorem}

\begin{proof}
First notice that the connection establishes a $G$-equivariant bundle
isomorphism $\varphi_{\mathcal{A}}:TM\rightarrow H(M)\oplus \nu$ given by
$v_{m}\mapsto (\mathrm{Hor}\,v_{m},\mathcal{A}(v_{m}))$ where $\mathrm{Hor}$ is the projection
onto $H(M)$ given by $\mathrm{Hor}\,v_{m}=v_{m}-(\mathcal{A}(v_m))_M(m)$. The inverse
of this map is then simply $(v_m,[\xi]_{m})\mapsto v_m+\xi_M(m)$ which
is clearly well defined. Next, observe that each of these bundles is
a saturated bundle over $M_H$. That is, we have
\begin{equation*}
TM=G\cdot TM|_{M_H}\simeq G\cdot \left(H(M)|_{M_H}\oplus \nu|_{M_H}\right)
\end{equation*}
since $G$ acts on $TM$ by tangent lifts and the base action satisfies $G\cdot M_{H}=M$.
Because of these saturations, it is enough to study the strata of the fiber
over some $m\in M_H$
since the strata over any other fiber will be $g$-translations. So, we first fix $m\in M_H$ and
consider the diagonal $H$-action on $H(M)_m\oplus \nu_m$.
Notice that $H(M)_m$ is isomorphic to a linear slice for the $G$-action on $M$ at the point $m$.
Since $M$ is a single orbit type manifold the $H$-action on the linear slice can only
have one orbit type, and therefore the entire space must be fixed by $H$ since
$H$ fixes $0$. Therefore, since the $H$-action is diagonal, we have reduced the study of the stratification of $TM$ to the study
of the strata of the $H$-action on $\nu_m$. Recall that $\nu|_{M_H}=M_H\times \mathfrak{g}/\mathfrak{h}$
and the action is given by, from equation \eqref{equation nu action}, $h\cdot [\xi]=[\operatorname{Ad}_h\xi]$ since $H$ fixes the base point. It follows that if
we denote by $(K)$, ($(K)\leq (H)$), the elements of the isotropy lattice for the action $H\times \mathfrak{g}/\mathfrak{h}\rightarrow \mathfrak{g}/\mathfrak{h}$, then this lattice is in one-to-one
correspondence with the isotropy lattice for the $G$-action on $H(M)\oplus \nu$. Furthermore, since $\varphi_{\mathcal{A}}:TM\rightarrow H(M)\oplus \nu$ is a $G$-equivariant bundle isomorphism, the isotropy lattice for the $G$-action on $TM$ is identical to that of $H(M)\oplus\nu$ and $\varphi_{\mathcal{A}}$ restricts to a smooth isomorphism,
\begin{equation*}
(TM)_{(K)}\simeq H(M)\times \nu_{(K)}.
\end{equation*}
From this isomorphism, it is now clear that the map $\mathcal{A}:TM\rightarrow \nu$ is a stratified morphism mapping each orbit type stratum $(TM)_{(K)}$ onto the orbit type stratum $\nu_{(K)}$ and covering the identity map on the base.

Denote the quotient map for the $G$-action on $M$ by $\pi:M\rightarrow M/G$. Of course $M/G$ is a smooth manifold since $M$ is a single orbit type manifold, and is in fact diffeomorphic to the orbit space for the free and proper action of $N(H)/H$ on $M_H$. Consider the restriction of the $G$-action to the subbundle $H(M)$. The quotient by this action is a manifold since $H(M)$ has just one orbit type. Furthermore, the quotient is a bundle over $M/G$ since the action covers the action of $G$ on $M$.
The fiber of this bundle over a point $[m]\in M/G$ with $m\in M_H$ is $H_m(M)/H\simeq T_mH(M)$ since the $H$-action fixes every point in $H_m(M)$ (by the argument given earlier). Therefore,  the isomorphism $T_{m}\pi:H_m(M)\rightarrow T_{[m]}(M/G)$ (since $\mathrm{Ker}\ T_{m}\pi\simeq\nu_m$)  induces a bundle isomorphism $H(M)/G\simeq T(M/G)$.
Similarly the bundle $\nu_{(K)}$ is a single orbit type space and its quotient is then a smooth manifold which is also a fiber bundle over $M/G$. The fiber of this bundle over a point $[m]\in M/G$ such that $m\in M_H$ is just the quotient by the $H$-action on $(\nu_{(K)})_m$ which is $(\mathfrak{g}/\mathfrak{h})_{(K)}/H$.

Next, since the $G$-action on the direct product bundle $H(M)\times \nu_{(K)}$ is diagonal, the product bundle has fixed orbit type which is $(K)$ and the quotient is therefore a manifold. Furthermore, since the action is by lifts, this quotient is a bundle over $M/G$ and its fiber over a point $[m]\in M/G$ with $m\in M_H$ is given by the quotient of the $H$-action on the fiber $H_m(M)\times_m \nu_{(K)}$. Since $H$ acts diagonally and fixes $H_m(M)$, the quotient is just $H_m(M)\times_m (\mathfrak{g}/\mathfrak{h})_{(K)}/H\simeq T_{[m]}(M/G)\times_m (\nu_{(K)}/G)_m$. It follows that the quotient, $(H(M)\times \nu_{(K)})/G$ is isomorphic to the direct product bundle over $M/G$ given by  $T(M/G)\times \nu_{(K)}/G$.

\end{proof}

\subsection{Singular Atiyah Sequence}
It is useful to describe a singular version of the standard Atiyah
sequence for a principal bundle. In this singular case, i.e. of a single orbit type manifold,
the analogous sequence is no longer a sequence of vector bundles over
the quotient space, but rather a sequence of stratified vector bundles and each
arrow corresponds to a stratified morphism. The singular connection establishes
a splitting of the sequence of vector bundles over $M$,
\begin{equation}
\label{short_exact}
0\rightarrow \nu \rightarrow TM\xrightarrow{\varphi_\mathcal{A}} H(M)\oplus\nu \xrightarrow {\mathrm{pr}} \pi^{\ast}T(M/G)\rightarrow 0
\end{equation}

where $\pi^{\ast}T(M/G)$ is the pull back bundle of $M\rightarrow M/G$ with respect to
the tangent projection $\tau_{M/G}:T(M/G)\rightarrow M/G$. That is,
\begin{equation*}
\pi^{\ast}T(M/G)=\left\{(m,v_{[m]})\, : \, \pi(m)=\tau_{M/G}(v_{[m]})\right\},
\end{equation*}
and $\mathrm{pr}(v_m)=(m,T_{m}\pi(v_m))$. Notice that
$\pi^{\ast}T(M/G)$ is a fiber bundle over $M$ with fiber over $m\in
M$, $T_{[m]}(M/G)$, and a fiber bundle over $T(M/G)$ with fiber over
$v_{[m]}$, the orbit $G\cdot m$ where $\pi(m)=\tau_{M/G}(v_{[m]})$.
Furthermore $\pi^{\ast}T(M/G)$ carries a proper $G$-action given by
$g\cdot (m,v_{[m]})=(g\cdot m,v_{[m]})$. It is easy to see that this
space has only a single orbit type, $(H)$, identical to the one for
the action of $G$ on $M$. Now, as in the regular case, the singular
connection $\mathcal{A}$ determines an injective horizontal lift map
for each $m\in M$, $\mathrm{hor}_{m}:T_{[m]}(M/G)\rightarrow T_{m}M$
which takes values in $H(M)_{m}$. Consequently there is an injective
bundle map $\pi^{\ast}T(M/G)\rightarrow TM$ which takes values in
$H(M)$ and establishes a $G$-equivariant isomorphism of bundles
$\pi^{\ast}T(M/G)\simeq H(M)$ given by $(m,v_{[m]})\mapsto
\mathrm{hor}_{m}(v_{[m]})$. Notice that the splitting of the
sequence is identical for all the strata and this fact depends
crucially on the property that the stratification of $TM$ is along
the vertical part of the Whitney sum in \eqref{short_exact}, the
bundle $\nu$. Each stratified sequence, with smooth morphisms is
just
\begin{equation*}
%\label{short_exact_strata}
0\rightarrow \nu_{(K)}\rightarrow (TM)_{(K)}\xrightarrow{\varphi_\mathcal{A}} H(M)\times \nu_{(K)}\rightarrow \pi^{\ast}T(M/G)\rightarrow 0.
\end{equation*}

Following the construction in the regular case, we consider the quotient of the stratified sequence \eqref{short_exact} to obtain a stratified split Atiyah sequence,
\begin{equation*}
%\label{atiyah}
0\rightarrow \nu/G\rightarrow (TM)/G\simeq T(M/G)\times \nu/G\rightarrow T(M/G)\rightarrow 0
\end{equation*}
with smooth strata and morphisms given by
\begin{equation*}
%\label{atiyah_strata}
0\rightarrow \nu_{(K)}/G\rightarrow (TM)_{(K)}/G\simeq T(M/G)\times\nu_{(K)}/G\rightarrow T(M/G)\rightarrow 0
\end{equation*}

\subsection{Curvature}
Recall that the usual theory of curvature begins with the definition of the covariant differential of a connection
$$D\mathcal{A}(u,v):=\mathbf{d}\mathcal{A}(\mathrm{Hor}\,u,\mathrm{Hor}\,v)$$ often denoted $B=DA$
and subsequent proof that this is a tensor, and that it verifies the identity
$$B(X,Y)=-\mathcal{A}([\mathrm{Hor}X,\mathrm{Hor}Y])$$ for vector fields $X$ and $Y$. This last identity gives the curvature the interpretation as the measure of non-integrability of the horizontal distribution of the connection $\mathcal{A}$. In our case the singular connection is a bundle map from $TM$ to $\nu$. Rather than define a covariant differential for this object, we make the following definition of its curvature form. In Remark \ref{ehresmann approach} we take an alternative approach of defining the covariant differential by realizing the singular connection as an Ehresmann connection.

\begin{definition}
\label{definition curvature}
The curvature of $\mathcal{A}:TM\rightarrow \nu$ is defined to be
\begin{equation}
\label{curvature}
B(u_m,v_m):=-\mathcal{A}\left([\mathrm{Hor}X,\mathrm{Hor}Y]\right)(m).
\end{equation}
where $X$ is a vector field extending $u_m$ and $Y$ is a vector field extending $v_m$. Recall that $\mathrm{Hor}:TM\rightarrow H(M)$ is the projection relative to the singular connection.
\end{definition}

\begin{proposition}
\label{proposition curvature}
The curvature $B$ given in the previous definition is well defined.
Also,
$B$ is a $G$-equivariant bundle map
$
B:\wedge^2TM\rightarrow \nu
$
and it takes values in the stratum of $\nu$ containing the zero section, $\nu_{(H)}$.
Furthermore, $B$ uniquely determines a reduced curvature form $\mathcal{B}:\wedge^2T(M/G)\rightarrow \nu_{(H)}/G$ which is a bundle map covering the identity in $M/G$.
\end{proposition}

\begin{proof}
We demonstrate that equation \eqref{curvature} uniquely determines a well defined $\nu$-valued 2-form on $M$
by showing that it is tensorial, i.e. that $B(fX,Y)=fB(X,Y)$.
Recall that $[fX,Y]=f[X,Y]-Y(f)X$ so that
\begin{align*}
B(fX,Y)&=-\mathcal{A}\left([f\mathrm{Hor}X,\mathrm{Hor}Y]\right)\\
&=-\mathcal{A}\left(f[\mathrm{Hor}X,\mathrm{Hor}Y]-\mathrm{Hor}Y(f)\mathrm{Hor}X\right)\\
&=-f\mathcal{A}\left([\mathrm{Hor}X,\mathrm{Hor}Y]\right)=fB(X,Y),
\end{align*}
as required.

Denote by $\varphi_{h}$ the diffeomorphism on $M$ corresponding to the group element $h\in G$.
To check $G$-equivariance of $B$, given $X$ and $Y$ vector fields extending $u_m,v_m\in T_{m}M$,
note that $(\varphi_{g^{-1}}^{\ast}X)(g\cdot m)=T_{m}\varphi_{g}(X(m))=g\cdot u_m$ and similarly for $\varphi_{g^{-1}}^{\ast}Y$. Therefore,
\begin{align*}
(\varphi_{g}^{\ast}B)(u_m,v_m)&=B(g\cdot u_m,g\cdot v_m)=-\mathcal{A}([\mathrm{Hor}(\varphi_{g^{-1}}^{\ast}X),\mathrm{Hor}(\varphi_{g^{-1}}^{\ast}Y)](g\cdot m))\\
&=-\mathcal{A}((\varphi_{g^{-1}}^{\ast}[\mathrm{Hor}X,\mathrm{Hor}Y])(g\cdot m))
=-\mathcal{A}(g\cdot([\mathrm{Hor}X,\mathrm{Hor}Y](m)))\\
&=-g\cdot\mathcal{A}([\mathrm{Hor}X,\mathrm{Hor}Y](m))=g\cdot B(u_m,v_m),
\end{align*}
as required.
This equivariance has the following consequence for the values of the curvature form.
Let $m\in M_H$ and $u_m,v_m\in T_{m}M$. Recall that since the manifold consists of a single orbit type, the $H$-action fixes $H(M)_m$.
Furthermore, since
\begin{equation*}
u_m=\mathrm{Hor}(u_m)+(\mathcal{A}(u_m))_{M}(m)=\mathrm{Hor}(u_m)+\mathrm{Ver}(u_m),
\end{equation*}
we have $h\cdot u_m=\mathrm{Hor}(u_m)+h\cdot \mathrm{Ver}(u_m)$. Now let $X$ extend $u_m$. Then
the vector field $\tilde{X}=\mathrm{Hor}X+\varphi_{h^{-1}}^{\ast}\mathrm{Ver}X$ extends $h\cdot u_m$ since
\begin{align*}
\tilde{X}(h\cdot m)&=\mathrm{Hor}X(h\cdot m)+(\varphi_{h^{-1}}^{\ast}\mathrm{Ver}X)(h\cdot m)\\
&=\mathrm{Hor}X(m)+h\cdot\mathrm{Ver}X(m)\\
&=h\cdot \mathrm{Hor}X(m)+h\cdot \mathrm{Ver}X(m)\\
&=h\cdot (\mathrm{Hor}X(m)+\mathrm{Ver}X(m))=h\cdot u_m.
\end{align*}

Following a similar construction for $v_m$ extended by $Y$, and a comparable definition of $\tilde{Y}$, we have
\begin{align*}
B(h\cdot u_m,h\cdot v_m)&=-\mathcal{A}([\mathrm{Hor}\tilde{X},\mathrm{Hor}\tilde{Y}](h\cdot m))\\
&=-\mathcal{A}([\mathrm{Hor}X,\mathrm{Hor}Y](m))\\
&=B(u_m,v_m).
\end{align*}

On the other hand, by $G$-equivariance of $B$, we have,
\begin{equation*}
B(h\cdot u_m,h\cdot v_m)=h \cdot B(u_m,v_m),
\end{equation*}
so that we are forced to conclude that for every $h\in H$,  $B(u_m,v_m)=h\cdot B(u_m,v_m)$
and therefore the curvature takes values in the $H$-fixed set of $\nu_{m}=\mathfrak{g}/\mathfrak{h}$, which is the fiber over $m$ of the stratum of $\nu$ containing $0_m$, that is, $\nu_{(K)}$.

For the reduced curvature form, let $u_{[m]},v_{[m]}\in T_{[m]}(M/G)$. We define
\begin{equation*}
%\label{reduced curvature}
\mathcal{B}(u_{[m]},v_{[m]}):=[B(u_m,v_m)],
\end{equation*}
where $T_{m}\pi (u_m)=u_{[m]}$, $T_{m}\pi (v_m)=v_{[m]}$ and
$[B(u_m,v_m)]$ denotes the element of $\nu_{(H)}/G$ determined by
$B(u_m,v_m)\in \nu_{(H)}$. An easy calculation using
$G$-equivariance of $B$ shows that this is well defined.
\end{proof}

\begin{remark}
\label{ehresmann approach} Alternatively, we can approach the
covariant derivative of the connection by realizing that a singular
connection is equivalent to an Ehresmann connection as follows. An
Ehresmann connection is simply a choice of horizontal distribution
complementary to the vertical distribution that is $G$-invariant.
Given a singular connection $\mathcal{A}$, one defines an Ehresmann
connection $\Gamma\in \Omega^1(M;V(M))$ (a $V(M)$-valued one form on
$M$, where $V(M)$ is the vertical distribution) by
$\Gamma(v_m)=[\sigma]_m\circ \mathcal{A}(v_m)$ where
$[\sigma]:\nu\rightarrow V$ is the bundle isomorphism induced by the
action $\sigma:\mathfrak{g}\rightarrow TM$. Conversely, given an
Ehresmann connection $\Gamma$ one induces the singular connection by
$\mathcal{A}(v_m)=[\sigma]^{-1}\circ \Gamma(v_m)$. Now, recall that
for $\lambda\in\Omega^{k}(M;V(M))$, a $V(M)$-valued $k$-form on $M$,
the definition of the covariant derivative of $\lambda$,
$D\lambda\in \Omega^{k+1}(M;V(M))$ is
\begin{align}
\label{covariant ehresmann}\displaystyle
\nonumber D\lambda(X_0,\dots, X_k):&=\sum_{i=0}^{k} (-1)^i[X_i^{\mathrm{hor}},\lambda(X_0^{\mathrm{hor}},\dots,\hat{X}_{i},\dots,X_k^{\mathrm{hor}})]^{\mathrm{ver}}\\
&+\sum_{0\leq i<j\leq k} (-1)^{i+j}\lambda([X_i^{\mathrm{hor}},X_j^{\mathrm{hor}}],X_0^{\mathrm{hor}},\dots, \hat{X}_{i},\dots,\hat{X}_j,\dots,X_k^{\mathrm{hor}})
\end{align}
where $X_0,\dots,X_K$ are vector fields on $M$, and $X^{\mathrm{hor}}$ and $X^{\mathrm{ver}}$ is the horizontal and vertical projection
of $X$.

We can then alternatively define
\begin{equation}
\label{covariant 2}
D\mathcal{A}:=[\sigma]\circ D\Gamma
\end{equation}
for $\Gamma=[\sigma]\circ\mathcal{A}$ which is consistent with
equation \eqref{curvature} since, using equation \eqref{covariant
ehresmann},
$\mathrm{curv}_{\Gamma}(X,Y):=D\Gamma(X,Y)=-\Gamma([X^{\mathrm{hor}},Y^{\mathrm{hor}}])$
which satisfies $B=[\sigma]\circ \mathrm{curv}_{\Gamma}$ where $B$
is the curvature of the singular connection $\mathcal{A}$ as defined
in equation \eqref{curvature}. Finally we remark that using
definition \eqref{covariant 2} the Bianchi identity for $B$ follows
immediately since $D\mathrm{curv}_{\Gamma}=0$ implies $DB=0$.
\end{remark}

\subsection{A holonomy theorem}

Our first result concerns the lowest dimensional stratum of the $\operatorname{Ad}$-induced $H$-action on $\mathfrak{g}/\mathfrak{h}$. As the next lemma shows, this stratum turns out to be the Lie algebra
of the group $N(H)/H$ which is precisely the group that acts freely on the submanifold
$M_{H}$. We will then establish a bundle reduction theorem that will enable us to prove the Ambrose-Singer theorem for singular connections: that the Lie algebra of the holonomy group for a singular connection is given by the image of the curvature of the connection. By Lemma \ref{zero stratum}, this holonomy group is then contained in the group $N(H)/H$.

\begin{lemma}
\label{zero stratum}
The stratum containing $0\in \mathfrak{g}/\mathfrak{h}$ corresponding to the stabilizer group $H$,
i.e. the lowest dimensional stratum for the $H$-induced stratification, is the Lie algebra of the group
$N(H)/H$.
\end{lemma}

\begin{proof}
By definition $(\mathfrak{g}/\mathfrak{h})_{(H)}=(\mathfrak{g}/\mathfrak{h})^{H}$ the fixed set
by the linear $H$-action. Let us denote this action by $h\cdot [\xi]$. We then have $h\cdot [\xi]=[\operatorname{Ad}_{h}\xi]$ and therefore
\begin{align*}
(\mathfrak{g}/\mathfrak{h})_{(H)}&=\left\{[\xi]\in\mathfrak{g}/\mathfrak{h}\ \, : \, \ [\operatorname{Ad}_{h}\xi]=[\xi]\mathrm{\ for\  all\ }h\in H\right\}\\
&=\{[\xi]\in\mathfrak{g}/\h\ \, : \,\operatorname{Ad}_{h}\xi-\xi\in \mathfrak{h}\mathrm{\ for\  all\ }h\in H\}.
\end{align*}
Next, we prove that the set $\{\xi\in\mathfrak{g}\, : \,\operatorname{Ad}_{h}\xi-\xi\in \mathfrak{h}\mathrm{\ for\  all\ }h\in H\}$ is in fact Lie$(N(H))$.
First suppose $\xi\in \mathrm{Lie}(N(H))$. Then we have $\exp (t\xi) h\exp (-t\xi)\in H$ for all $t$ and therefore $h^{-1}\exp (t\xi) h\exp (-t\xi)\in H$ for all $t$ which is a curve passing through $e$ at $t=0$.
Therefore we have
\begin{equation*}
\ddt h^{-1}\exp (t \xi) h \exp (-t\xi) =\operatorname{Ad}_{h^{-1}}\xi - \xi \in \mathfrak{h},
\end{equation*}
as required.
Conversely, suppose $\xi \in \mathfrak{g}$ satisfies
\begin{equation}
\label{equation Ad_minus_Id}
\operatorname{Ad}_{h}\xi-\xi \in \mathfrak{h} \mathrm{\ for\ all}\ h\in H.
\end{equation}
We need to show that $\exp (t\xi) h \exp (-t\xi)\in H$ for all $t$. Notice that, differentiating equation
\eqref{equation Ad_minus_Id} at $e$ in the direction $\eta\in \mathfrak{h}$, we have
$[\eta,\xi]\in \mathfrak{h}$. Now, because $\exp$ intertwines the $\operatorname{Ad}$ action
with the $\operatorname{AD}$ action (the action by inner automorphisms of $G$ on itself) we have
\begin{equation*}
h\exp(t\xi) h^{-1}\exp (-t\xi)=(\operatorname{AD}_{h}\exp (t\xi)) \exp (-t\xi)=\exp (t \operatorname{Ad}_{h}\xi)\exp (-t\xi).
\end{equation*}
We compute this last expression using the Baker-Campbell-Hausdorff formula as follows:
$\exp (t\operatorname{Ad}_{h}\xi) \exp(-t\xi)=\exp (t [\operatorname{Ad}_{h}\xi - \xi +\mathcal{O}(t^2)])$,
where $\mathcal{O}(t^2)$ is a convergent series in $\mathfrak{g}$ each term of which is a composition of brackets containing a
bracket of $\operatorname{Ad}_{h}\xi$ or $\xi$ with $[\operatorname{Ad}_{h}\xi,\xi]$.
However, $[\operatorname{Ad}_{h}\xi,\xi]=[\operatorname{Ad}_{h}\xi-\xi,\xi]\in \mathfrak{h}$ so if
we take a bracket of $[\operatorname{Ad}_{h}\xi,\xi]$ with $\xi$, we again get an element in $\mathfrak{h}$.
For the other type of term, involving a bracket of $\operatorname{Ad}_{h}\xi$ with $[\operatorname{Ad}_{h}\xi,\xi]$, we have just seen that $[\operatorname{Ad}_{h}\xi,\xi]\in \mathfrak{h}$.
But, if $\eta'\in\h$, $$[\operatorname{Ad}_{h}\xi,\eta']=[\operatorname{Ad}_{h}\xi-\xi +\xi,\eta']=[\eta,\eta']+[\xi,\eta']\in \mathfrak{h}$$ since $\eta=\operatorname{Ad}_h\xi-\xi \in\mathfrak{h}$ by \eqref{equation Ad_minus_Id}.
It follows that $\mathcal{O}(t^2)\in \mathfrak{h}$ and therefore $h\exp (t\xi) h^{-1}\exp (-t\xi)\in H$. Multiplying on the left by $h^{-1}$ we conclude that $\exp( t\xi) h \exp (-t\xi) \in H$ as required.

\end{proof}
The next theorem addresses bundle reduction to a principal bundle for single orbit type manifolds and establishes
a one-to-one correspondence with singular connections and principal bundle
connections of the reduced bundle.
\begin{theorem}
\label{bundle reduction}
Consider the following commutative diagram.
\[
\xymatrix{
M_{H}\ar[r]^{=}\ar[dd]_{\pi'}&M_{H}\ar[r]^{\iota}\ar[dd]^{\pi_{N(H)}}&M\ar[dd]^{\pi}\\
\\
M_{H}/(N(H)/H)\ar[r]^{=}&M_{H}/(N(H))\ar[r]^{\simeq}&M/G}
\]
Then $\pi_{N(H)}:M_{H}\rightarrow M_{H}/N(H)$ is a bundle reduction of $\pi:M\rightarrow M/G$
and $\pi':M_{H}\rightarrow M_{H}/(N(H)/H)$ is a \emph{principal} bundle reduction of $M_{H}\rightarrow M_{H}/N(H)$.

There is a one-to-one correspondence between principal bundle connections on $M_{H}\rightarrow M_{H}/(N(H)/H)$ and singular connections on $M$. Furthermore, the curvature form for the singular connection restricted to $M_H$ is equal to the curvature form $\Omega$ of the principal connection.

\end{theorem}
\begin{proof}
As we have already remarked, every $G$-orbit in $M$ intersects $M_H$
in a unique $N(H)$-orbit. Therefore for $m\in M_H$,
$\pi_{N(H)}(m)=\pi(\iota(m))$, and then $M/G\simeq M_H/N(H)$.
Furthermore, since $H$ fixes every point of $M_H$, the free action
of $N(H)/H$ on $M_H$ has the same orbits as the action of $N(H)$ on
$M_H$ and therefore the bundle $M_H\rightarrow M_H/(N(H)/H)$ is a
principal bundle reduction of $M_H\rightarrow M_H/N(H)$. To set up
the one-to-one correspondence of connections, we start with a
connection on the principal bundle $\pi':M_H\rightarrow
M_H/(N(H)/H)$. Denote this connection by $\Gamma$ consisting of
horizontal spaces $Q_m$ for each $m\in M_H$ and its connection form
by $\omega:TM_{H}\rightarrow \mathfrak{n}/\mathfrak{h}$. We show how
to induce from this data a connection on $M\rightarrow M/G$ where
the connection data will consist of a horizontal distribution
invariant with respect to the group action. Given $m'\in M$ we have
$m'=g\cdot m$ for some $m\in M_H$ and $g\in G$. Denoting as before
the action of $G$ on $M$ by $\varphi$, and by slight abuse of
notation, denoting the restricted action of $N(H)$ on $M_H$ also by
$\varphi$, take
\begin{equation}
\label{small to big connection} H_{m'}:=T_m \varphi_{g}\circ
T_{m}\iota(Q_{m}).
\end{equation}
To check it is well defined, take another realization, $m'=g_1\cdot
m_1$ where $m_1\in M_H$. Then, $g\cdot m=g_1\cdot m_1$ so that
$g_{1}^{-1}g\cdot m=m_1$ and therefore $h:=g_{1}^{-1}g\in N(H)$.
Now, by the $N(H)$-invariance of the connection on $M_H$ we know
that $Q_{m_1}=T_{m}\varphi_{h}\cdot Q_{m}$ and of course for $h\in
N(H)$, $\iota\circ\varphi_{h}=\varphi_{h}\circ\iota$. Therefore,
\begin{align*}
T_{m_1}\varphi_{g_1}\circ T_{m_1}\iota
(Q_{m_1})&=T_{m_1}\varphi_{g_1}\circ T_{h\cdot m}\iota \circ
T_m\varphi_{h} (Q_{m}) =T_m\varphi_{g_1h}\circ T_m\iota
(Q_m)=T_m\varphi_g\circ T_m\iota (Q_m),
\end{align*}
proving that equation \eqref{small to big connection} is well defined. We next establish that these
horizontal spaces are complementary to the vertical spaces in $M\rightarrow M/G$.
For any $m'\in M$, represented by $m'=g\cdot m$ for $m\in M_H$, we have the following commutative diagram,
\[
\xymatrix{
Q_m\ar[r]^{T_{m}\phi_{g}\circ T_{m}\iota}\ar[dd]_{T_{m}\pi'} & H_{m'}\ar[dd]^{T_{m'}\pi}\\
\\
T_{\pi'(m)}(M_{H}/(N(H)/H)\ar[r]^{\simeq} & T_{\pi(m')}(M/G).}
\]
Since $T_{m}\pi':Q_{m}\rightarrow T_{\pi'(m)}(M_{H}/(N(H)/H))$ is an
isomorphism and the diagram commutes, we must have that $T_{m'}\pi:
H_{m'}\rightarrow T_{\pi(m')}(M/G)$ is also an isomorphism, proving
that $H_{m'}$ is a complement to $\mathrm{ker}\,T_{m'}\pi$, so that
the all the $H_m$ spaces define a smooth distribution $H(M)$
transversal to the $G$-action on $M$. By \eqref{small to big
connection} this distribution is $G$-invariant. Therefore it defines
 a singular connection on $M$ with corresponding
connection form $\mathcal{A}:TM\rightarrow \nu$, from the
isomorphism $TM\simeq H(M)\oplus \nu$, given by
$$
\mathcal{A}(v)=\mathcal{P}(v)
$$
where $\mathcal{P}:TM\rightarrow \nu$ is the projection induced from the splitting.

Conversely, starting with a singular connection $\mathcal{A}:
TM\rightarrow \nu$ we induce a principal connection on
$M_{H}\rightarrow M_H/(N(H)/H)$. As before, let
$H_m:=\mathrm{Ker}\,\mathcal{A}(m)$. For $m\in M_H$, in fact
$H_m\subset T_{m}M_{H}$. To see this, recall that $M=G\cdot M_H$ so
that, for $m\in M_H$ it follows that $T_{m}M=T_{m}M_H +
\mathfrak{g}\cdot m$. We refine this by taking an arbitrary
complement, $\mathfrak{r}$, of $\mathfrak{n}:=\mathrm{Lie}(N(H))$ in
$\mathfrak{g}$ so that $\mathfrak{g}=\mathfrak{n}\oplus
\mathfrak{r}$. Now, since $\xi_{M}(m)\in T_{m}M_{H}$ if and only if
$\xi\in \mathfrak{n}$, we have $T_{m}M=T_{m}M_H\oplus
\mathfrak{r}\cdot m$. On the other hand, since $T_{m}M=H_m\oplus
\mathfrak{g}\cdot m$ it follows that $H_m\subset T_{m}M_H$. In fact
we get the finer splitting $T_{m}M=H_m\oplus \mathfrak{n}\cdot
m\oplus \mathfrak{r}\cdot m$ with $T_{m}M_{H}=H_m\oplus
\mathfrak{n}\cdot m$. Finally notice that since the distribution
defined by the spaces $H_m$ is $G$-invariant, then it is also
$N(H)$-invariant so that we can define the horizontal spaces on
$TM_H$ by $Q_{m}=H_m$ for $m\in M_H$.

To describe the corresponding induced connection form, notice that
$\iota^{\ast}\mathcal{A}:TM_{H}\rightarrow \iota^{\ast}\nu$ where
$\iota^{\ast}\nu$  is the pull back bundle of $\nu$ with respect to
the map $\iota:M_H\rightarrow M$, which is just the restriction of
the bundle $\nu$ to the base $M_H$. By Proposition \ref{properties
of nu}, this bundle is simply $M_H\times \mathfrak{g}/\mathfrak{h}$.
However $\iota^{\ast}\mathcal{A}$ only takes values in $M_H\times
\mathfrak{n}/\mathfrak{h}$ since $\xi_{M}(m)\in T_{m}M_H$ if and
only if $\xi\in \mathfrak{n}$. Clearly $\iota^{\ast}\mathcal{A}$ is
onto $M_H\times\mathfrak{n}/\mathfrak{h}$ and induces a map
$\omega:TM_{H}\rightarrow \mathfrak{n}/\mathfrak{h}$. This map is
$(N(H)/H)$-equivariant and it is easy to check that $\mathrm{Ker}\
\omega(m)=Q_m$, so that $\omega$ is the corresponding principal
connection form of the induced connection.

For the curvature correspondence, let $\Omega$ be the
$\mathfrak{n}/\mathfrak{h}$-valued curvature form of $\omega$ on
$M_H$. Let $B$ be the curvature of $\mathcal{A}$ as defined in
Definition \ref{definition curvature}. For $m\in M_H$, given
$u_m,v_m\in T_{m}M$, let $\tilde{u}_m:=\mathrm{Hor}\ u_m$, and
$\tilde{v}_m:=\mathrm{Hor}\ v_m$. Using the definition of $B$
\eqref{curvature}, let $X$ and $Y$ be arbitrary extensions of
$\tilde{u}_m$, $\tilde{v}_{m}$ chosen to be tangent to $M_H$.  In
fact we can construct these vector fields by first projecting $u_m,
v_m$ to the quotient $M/G$ and then taking arbitrary extensions in
$M/G\simeq M_H/(N(H)/H)$, denoted $\tilde{X}, \tilde{Y}$. Next, lift
the vector fields horizontally with respect to the principal bundle
$\pi':M_H\rightarrow M_H/(N(H)/H)$, and then extend them to the
entire manifold by the $G$-action, which we can do since $M$ is
saturated by $M_H$. Since the $G$-action preserves the horizontal
distribution, this will define globally two horizontal vector fields
on $M$ denoted by $X$ and $Y$, extending $\tilde{u}_m$ and
$\tilde{v}_m$. Since they are horizontal, and the horizontal
distribution, restricted to $M_H$ is contained in $TM_H$, their
restrictions to $M_H$ are smooth vector fields on $M_H$. We then
have
\begin{align*}
B(u_m,v_m)&=B(\tilde{u}_m,\tilde{v}_m)=-\mathcal{A}([\mathrm{Hor} X,\mathrm{Hor} Y])(m)\\
&=-\omega([\mathrm{Hor} X,\mathrm{Hor} Y])(m)
=\Omega(m)(\tilde{u}_m,\tilde{v}_m)
\end{align*}
where for the third equality, we use the fact that the Jacobi bracket of $\mathrm{Hor} X$ and $\mathrm{Hor} Y$ in $M$ evaluated at $m\in M_H$ coincides with their bracket as vector fields in $M_H$ (evaluated at $m$) so that the third equality holds.  The final equality is just a consequence of the usual curvature identity for principal bundles.
\end{proof}

\begin{theorem}[Ambrose-Singer]
\label{Ambrose Singer}
The horizontal lift of a curve in the base space $M/G$ through a point $m\in M_{H}$ lies entirely
in $M_{H}$ and the holonomy group through $m$, $\mathcal{H}(m)$, of the singular connection  is contained in $N(H)/H$.  The Lie algebra of the holonomy group through $m$ is given by the image of the curvature of the singular connection at $m$.

\end{theorem}

\begin{proof}
Given $[m]=\pi(m)\in M/G$ and a loop $l(t)$ in $M/G$ through $[m]$, consider $m\in M_H$ such that $\pi(m)=[m]$.
From Proposition \ref{bundle reduction} , we can view the loop $l$ as a loop in the base manifold
$M_H/(N(H)/H)$. Consider its horizontal lift, $L(t)$ in the usual sense of principal bundles, through the point $m\in M_H$. Since, by the proof of Theorem \ref{bundle reduction}, the horizontal spaces of the bundle $M_H\rightarrow M_{H}/(N(H)/H)$ are mapped into
horizontal spaces of $M\rightarrow M/G$, it follows that $L(t)$ is a horizontal curve in $M$ for the singular connection through the point
$m$ and projects to $l(t)$ by construction. However, the horizontal lift of a curve in $M/G$ to $M$, through a specified point $m$ is unique. To prove this, one uses approximately the same argument as in the free case. Suppose $L_2(t)$ is another curve through $m$ which projects to $l(t)$. It follows that $L_{2}(t)=g(t)L(t)$ for some curve $g(t)\in G$. This curve is not unique since the action is not free. However, $\ddt L_2(t)=\dot{g}(t)L(t)+g(t)\dot{L}(t)$ by the Leibniz identity, and therefore, applying the connection
to this expression one finds $0=\A\left(\ddt L_2(t)\right)=\A(\dot{g}(t)L(t))$ since $\dot{L}(t)$ is horizontal. The only way $\dot{g}(t)L(t)$ can be horizontal is if it is zero and therefore $g(t)$ must be a curve that lies for all time in the stabilizer of $L(t)$ and therefore $L_2(t)=L(t)$ as required.

Next, we show that the Lie algebra of the holonomy group $\mathcal{H}(m)$ is the image of the curvature of the singular connection at $m$.
Since the horizontal lift of a loop in $M/G$ through $m$ lies entirely in $M_H$, then $\mathcal{H}(m)\in N(H)/H$. Now we can apply the standard holonomy theorem to the principal bundle $M_H\rightarrow M_H/(N(H)/H)$ with the unique principal connection induced from the singular connection. The conclusion is that the Lie algebra of the holonomy group is given by the curvature of this induced principal connection at $m$. However, the curvature of $B$, evaluated at $m$ coincides with the curvature of the induced principal connection. Since this can be done for any $m$ the statement follows.
\end{proof}

\section{Singular Sternberg}
The singular connection form $\mathcal{A}$ allows us to write down a $G$-equivariant
diffeomorphism,
\begin{equation}
\label{equivariant diffeo}
\phi_{\mathcal{A}}:\tilde{M}\times \nu^{\ast}\rightarrow T^{\ast}M
\end{equation}
which will play the fundamental role in establishing the connection dependent realization of
the Poisson stratified space $T^{\ast}M/G$. The mapping \eqref{equivariant diffeo} is a fiber product factorization of phase space into zero momentum and nonzero momentum fibers respectively.
We call the quotient of the domain in \eqref{equivariant diffeo}, $\tilde{M}\times_{G} \nu^{\ast}$ the singular Sternberg space, as it generalizes the original representation due to Sternberg of $(T^{\ast}M)/G$ in the free category. Using $\phi_{\mathcal{A}}$ we will also determine the minimal coupling forms on the strata of the symplectic quotients.

The construction of the Sternberg space reviewed in Section 2 generalizes to single orbit type
manifolds as follows.
We start by constructing the zero momentum space $\tilde{M}$ by taking the pull
back of $M\rightarrow M/G$ by the cotangent projection $\tau_{M/G}:T^{\ast}(M/G)\rightarrow M/G$ so that we have
\[
\xymatrix{
\tilde{M}\ar[r]^{\tilde{\tau}}\ar[d]_{\pi^{\sharp}}&M\ar[d]^{\pi}\\
T^{\ast}(M/G)\ar[r]^{\tau_{M/G}}&M/G}.
\]

We then have
\begin{proposition}
$\tilde{M}$ is a fiber bundle over $T^{\ast}(M/G)$ whose fibers are the $G$-orbits
of $M$. Furthermore, $\tilde{M}$ is a $G$-space with single orbit type and as a bundle over $M$ is
bundle isomorphic to the zero momentum space, $\mathbf{J}^{-1}(0)\simeq V(M)^\circ$ where
$V(M)^\circ$ is the annihilator of the vertical fibers in $M\rightarrow M/G$. Then, $\tilde{M}$ naturally inherits a singular connection form given by $\tilde{A}:=\tilde{\tau}^{\ast}\mathcal{A}:T\tilde{M}\rightarrow \nu$.
 \end{proposition}

\begin{proof}
By definition, $\tilde{M}=\{(\alpha_{[m]},m)\,:\, \pi(m)=\tau_{M/G}(\alpha_{[m]})=[m]\}$, and therefore, for
$\alpha_{[m]}\in T^{\ast}_{[m]}(M/G)$, $(\pi^{\sharp})^{-1}(\alpha_{[m]})=\{(\alpha_{[m]},m')\,:\,m'\in G\cdot m\}\simeq G\cdot m$ where $m$ is any representative
of $[m]$. Similarly, the fibers of $\tilde{\tau}$ satisfy $\tilde{\tau}^{-1}(m)\simeq T^{\ast}_{[m]}(M/G)$. Notice that $\tilde{M}$ inherits a $G$-action defined by $g\cdot (\alpha_{[m]},m)=(\alpha_{[m]},g\cdot m)$, clearly well defined, smooth, and proper since the action of $G$ on $M$ is so. Furthermore, since
$G_{(\alpha_{[m]},m)}=G_{m}$ it follows that $\tilde{M}$ is a single orbit type manifold with orbit type $(H)$, the same as for $M$, and therefore the quotient of $\tilde{M}$ by this action is a smooth manifold. By inspection this quotient is simply $T^{\ast}(M/G)$.
Consider the map $\psi:\tilde{M}\rightarrow T^{\ast}M$ given by $(\alpha_{[m]},m)\mapsto (T_{m}\pi)^{\ast}\alpha_{[m]}$. This map takes the fiber of $\tilde{M}$ over $m$ into $T^{\ast}_{m}M$ and takes values in $V(M)^\circ$ since $\left\langle (T_{m}\pi)^{\ast}\alpha_{[m]},\xi_{M}(m)\right\rangle=
\left\langle \alpha_{[m]},T_{m}\pi(\xi_{M}(m))\right\rangle =0$. Since the map $\psi$ restricted to $\tilde{\tau}^{-1}(m)$ is injective (being the dual to the surjective map $T_{m}\pi:T_{m}M\rightarrow T_{[m]}(M/G)$),
and since $\mathrm{dim}\,V(M)^\circ=\mathrm{dim}\,M-\mathrm{dim}\,G+\mathrm{dim}\,H=\mathrm{dim}\,(M/G)$ it follows that $\psi|_{\tilde{\tau}^{-1}(m)}:T^{\ast}_{\alpha_{[m]}}(M/G)\rightarrow V(M)^\circ_{m}$ is an isomorphism and therefore $\psi:\tilde{M}\rightarrow V(M)^\circ$ is a bundle isomorphism covering the identity on $M$. Finally, consider the map $\tilde{\tau}^{\ast}\mathcal{A}:T\tilde{M}\rightarrow \nu$. First notice that $\nu$ is the correct target bundle for a singular connection on $\tilde{M}$ since it is a single orbit type manifold with
isotropy $(H)$. Also, we have for each $\xi\in \mathfrak{g}$,
\begin{equation*}
\tilde{\tau}^{\ast}\mathcal{A}(\xi_{\tilde{M}}(\alpha_{[m]},m))=\mathcal{A}(\xi_M(m))=[\xi]_m.
\end{equation*}
Equivariance follows from equivariance of the projection $\tilde{\tau}$ and equivariance of $\mathcal{A}$.
\end{proof}

Denote by $\tilde M\times\nu^*$ the corresponding product bundle
over $M$. We will use the following notation for bundle projections,
$\pi_{\nu^{\ast}}:\tilde{M}\times \nu^{\ast}\rightarrow \nu^{\ast}$
and $\pi_{\tilde{M}}:\tilde{M}\times \nu^{\ast}\rightarrow
\tilde{M}$. Note then, that
$\pi^{\sharp}\circ\pi_{\tilde{M}}:\tilde{M}\times\nu^{\ast}\rightarrow
T^{\ast}(M/G)$ which we abbreviate, below, as simply $\pi^{\sharp}$
to economize notation.

\begin{remark}
Recall that, by definition, the momentum map for the $G$-action on $M$ is a map, $\mathbf{J}:T^{\ast}M\rightarrow \g^{\ast}$. However, since the image of $\mathbf{J}|_{T^{\ast}_mM}$ is
$\g_{m}^{\circ}\simeq \nu_{m}^{\ast}$, we can also regard $\mathbf{J}$ as a bundle map $T^{\ast}M\rightarrow \nu^{\ast}$ over $M$. In the following we will sometimes implicitly use this point of view.
\end{remark}

\begin{proposition}
\label{proposition phi_A} Associated to the singular connection
$\mathcal{A}$ there is an equivariant bundle isomorphism,
$\phi_{\mathcal{A}}:\tilde{M}\times \nu^{\ast}\rightarrow T^{\ast}M$
covering the identity on $M$ given by
\begin{equation*}
%\label{phi_A}
\phi_{\mathcal{A}}((\alpha_{[m]},m),\nu_m)=(T_{m}\pi)^{\ast}\alpha_{[m]}+\mathcal{A}^{\ast}(m)\nu_m.
\end{equation*}
Furthermore, the pull back of the canonical one-form on $T^{\ast}M$ by $\phi_{\mathcal{A}}$ is given by,
\begin{equation}
\label{pull back theta}
\phi_{\mathcal{A}}^{\ast}\Theta_{M}=(\pi^{\sharp})^{\ast}\Theta_{M/G}+\left\langle \mathbf{J}\circ\phi_{\mathcal{A}},\tilde{A}\right\rangle,
\end{equation}
where $\left\langle\cdot,\cdot\right\rangle$ indicates the natural pairing on the fibers of $\nu^{\ast}$ with
the corresponding fibers on $\nu$, $$\mathbf{J}\circ\phi_{\mathcal{A}}:\tilde{M}\times\nu^{\ast}\rightarrow \nu^{\ast}$$ is just the bundle projection so that $\left\langle\mathbf{J}\circ\phi_{\mathcal{A}},\tilde{\mathcal{A}}\right\rangle$ is the one-form on $\tilde{M}\times \nu^{\ast}$ given by
\begin{equation*}
\left\langle
\mathbf{J}\circ\phi_{\mathcal{A}},\tilde{\mathcal{A}}\right\rangle(v_{\lambda})=
\left\langle \nu_m,\mathcal{A}(m)(T_{\pi_{\tilde M}(\lambda)}\tilde
\tau(T_\lambda\pi_{\tilde M}(v_{\lambda}))\right\rangle,
\end{equation*}
with $\lambda=((\alpha_{[m]},m),\nu_m)$.
\end{proposition}
\begin{proof}
Equivariance of $\phi_{\mathcal{A}}$ follows immediately from the equivariance of the singular connection $\mathcal{A}$ and the definition of the $G$-action on $\tilde{M}\times\nu^\ast$.

Next, observe that for any
$\xi\in \mathfrak{g}$ we have
\begin{equation*}
\left\langle\mathbf{J}\circ\phi_{\mathcal{A}}((\alpha_{[m]},m),\nu_m),\xi\right\rangle
=\left\langle (T_{m}\pi)^{\ast}\alpha_{[m]}+\mathcal{A}^{\ast}(m)\nu_{m},\xi_{M}(m)\right\rangle
=\left\langle \nu_m,[\xi]_m\right\rangle=\left\langle \nu,\xi\right\rangle,
\end{equation*}
where $\nu\in (\mathfrak{g}_m)^\circ$
and therefore $\mathbf{J}\circ\phi_{\A}$, restricted to the fiber over $m$ takes values in the fiber
of $\nu$ over $m$, and is surjective on this fiber since it is the dual of the injective
map $[\xi]_{m}\mapsto \xi_{M}(m)$.
Let $v_{\lambda}\in T_{\lambda}(\tilde{M}\times \nu^{\ast})$ with $\lambda=((\alpha_{[m]},m),\nu_m)$. We have

\begin{align*}
(\phi_{\A}^{\ast}\Theta_{M})(v_{\lambda})&=\Theta_{M}(T_{\lambda}\phi_{\A}(v_\lambda))\\
&=\left\langle (T_m\pi)^\ast\alpha_{[m]}+\A^{\ast}(m)\nu_m,T_{\phi_{\A}(\lambda)}\tau_M\circ T_{\lambda}\phi_{\A}(v_{\lambda})\right\rangle\\
&=\left\langle \alpha_{[m]},T_{\lambda}(\pi\circ\tau_{M}\circ\phi_{\A})(v_{\lambda})\right\rangle +\left\langle \nu_m,\A(m)(T_{\lambda}(\tau_{M}\circ\phi_{\A})(v_{\lambda})\right\rangle\\
&=(\pi^{\sharp})^{\ast}\Theta_{M/G}(v_{\lambda})+\left\langle \mathbf{J}\circ\phi_{\A},\tA \right\rangle(v_{\lambda})
\end{align*}
using the facts that
$\pi\circ\tau_M\circ\phi_{\A}=\tau_{M/G}\circ\pi^{\sharp}$ and
$\tau_M\circ\phi_{\A}=\tilde{\tau}\circ\pi_{\tilde M}$ in the third
equality.
\end{proof}

In the next two theorems, we prove that the minimal coupling form due to Sternberg generalizes to the singular setting. Care must be taken to prove the generalization since we need to deal with a fiber product bundle, $\tilde{M}\times\mathcal{\tilde O}$ and not just the product of manifolds as in the free case. The proof of the extension to the singular setting will make repeated use of the fact that the bundles over $M$, $\tilde{M}$ and $\mathbf{J}^{-1}(\Orbit)$, are each  $G$-saturated fiber bundles over $M_H$.

\begin{theorem}
Let $\Orbit$ be a coadjoint orbit through a point in the image of $\mathbf{J}:T^{\ast}M\rightarrow \g^{\ast}$.
The map $\phi_{\mathcal{A}}$ restricts to  an equivariant
stratified bundle isomorphism,
$\mathbf{J}^{-1}(\mathcal{O})\rightarrow \tilde{M}\times \tilde{\Orbit}$ where $\tilde{M}\times \tilde{\Orbit}$ is the fiber product bundle over $M$, and $\tilde{\Orbit}\subset \nu^{\ast}$ is a sub-bundle with fiber over $m\in M_H$  given by $\Orbit\cap \mathfrak{h}^\circ$. The orbit type strata of $\tilde{M}\times\tilde{\Orbit}$ are $\tilde{M}\times \tilde{\mathcal{O}}_{(K)}$ where $K$ are the subgroups of $H$
determined by the isotropy lattice for the action $H\times \mathfrak{h}^\circ\rightarrow \mathfrak{h}^\circ$ and $(K)$ denote their conjugacy classes in $G$.
Denote by $\iota_{(K)}:\tilde{M}\times\tilde{\Orbit}_{(K)}\hookrightarrow \tilde{M}\times \nu^{\ast}$ the inclusion.
The
restriction of the form $\phi_{\mathcal{A}}^{\ast}\omega_M$, the pull back of the canonical symplectic form on $T^{\ast}M$, to each stratum $\tilde{M}\times \tilde{\mathcal{O}}_{(K)}$
is given by
\begin{equation}
\label{restricted omega}
\iota_{(K)}^{\ast}\phi_{\mathcal{A}}^{\ast}\omega_M=(\pi^{\sharp})^{\ast}\omega_{M/G}-\mathbf{d}\left\langle\mathbf{J}_{\tilde{\mathcal{O}}_{(K)}},\tilde{\mathcal{A}}\right\rangle,
\end{equation}
where $\omega_{M/G}$ is the canonical symplectic form on $T^*(M/G)$ and  $\mathbf{J}_{\tilde{\mathcal{O}}_{(K)}}$ is the restriction of $\mathbf{J}\circ\phi_{\mathcal{A}}$ to
$\tilde{M}\times\tilde{\mathcal{O}}_{(K)}$. Define the two form on $\tilde{M}\times \tilde{\Orbit}_{(K)}$
\begin{equation}
\label{the basic form}
\omega^{\Orbit,(K)}:=(\pi^{\sharp})^{\ast}\omega_{M/G}-\mathbf{d}\left\langle\mathbf{J}_{\tilde{\mathcal{O}}_{(K)}},\tilde{\mathcal{A}}\right\rangle-\mathbf{J}_{\tilde{\mathcal{O}}_{(K)}}^{\ast}\omega_{\mathcal{O}}^{+},
\end{equation}
where $\omega_{M/G}$ is the canonical symplectic form on $T^{\ast}(M/G)$, and $\omega_{\mathcal{O}}^{+}$ is the $(+)$ orbit symplectic form
on $\mathcal{O}$.
The two-form $\omega^{\Orbit,(K)}$ satisfies the following:
\begin{itemize}
\item[(i)] It is basic, i.e. $G$-invariant and annihilates $G$-vertical vectors.
\item[(ii)] It drops to a unique two-form $\omega^{\Orbit,(K)}_\mathrm{min}$ on $\tilde{M}\times_{G}\tilde{\mathcal{O}}_{(K)}$.
\item[(iii)] Denoting by $\pi_{(K)}:\tilde{M}\times\tilde{\mathcal{O}}_{(K)}\rightarrow \tilde{M}\times_{G}\tilde{\mathcal{O}}_{(K)}$ the orbit map, the reduced form $\omega^{\Orbit,(K)}_\mathrm{min}$ on $\tilde{M}\times_{G}\tilde{\mathcal{O}}_{(K)}$
 (defined by $\pi_{(K)}^{\ast}\omega^{\Orbit,(K)}_\mathrm{min}=\omega^{\Orbit,(K)}$) satisfies
\begin{equation}
\label{reduced_form}
\iota_{(K)}^{\ast}\phi_\mathcal{A}^*\omega_M-\mathbf{J}_{\tilde{\Orbit}_{(K)}}^{\ast}\omega_{\mathcal{O}}^{+}=\pi_{(K)}^{\ast}\omega^{\Orbit,(K)}_\mathrm{min}.
\end{equation}

\end{itemize}

\end{theorem}
\begin{proof}
First recall that $\tilde{M}$ is a single orbit type manifold with orbit type $(H)$ identical to that of $M$.
By definition, $\tilde{\Orbit}$ is the bundle over $M$ whose fiber over $m$ is $\nu_{m}^{\ast}\cap \Orbit$. It is clear that $(\mathbf{J}\circ\phi_{\A})^{-1}(\Orbit)=\tilde{M}\times\tilde{\Orbit}$ since the restriction of
$\mathbf{J}\circ \phi_{\A}$ to each fiber is given simply by $\pi_{\nu^{\ast}}$.
It follows that $\mathbf{J}_{\tilde{\Orbit}}((\alpha_{[m]},m),\nu_m)=\nu_m\in \Orbit\cap\mathfrak{g}_m^\circ\subset \Orbit$.

On the other hand, the bundle $\tilde{\Orbit}$ over $M$
has the structure of a stratified fiber bundle which also trivializes over $M_H$. Restricted to $M_H$ we have
\begin{equation*}
\tilde{ \Orbit}|_{M_H}=M_{H}\times  (\Orbit\cap \mathfrak{h}^\circ)
\end{equation*}
and furthermore $\tilde{ \Orbit}=G\cdot \tilde{ \Orbit}|_{M_H}$ since the fiber of $\tilde{O}$ over
the point $g\cdot m$ in $M$ is just $\Orbit_{\mu}\cap (\operatorname{Ad}_g\mathfrak{h})^\circ=\Orbit_{\mu}\cap g\cdot \nu^{\ast}_m$ according to the definition of the $G$-action on $\nu^{\ast}$.

Equation \eqref{restricted omega} follows by taking the restriction
of equation \eqref{pull back theta} of the previous Proposition
\ref{proposition phi_A} to the stratum $\tilde{M}\times
\tilde{\Orbit}_{(K)}$ and then taking the exterior derivative. For
$(i)$, note that $\phi_{\mathcal{A}}$ is $G$-equivariant, and
$\omega_{M}$ is $G$-invariant and therefore
$\phi_{\mathcal{A}}^{\ast}\omega_{M}=(\pi^{\sharp})^{\ast}\omega_{M/G}-\mathbf{d}\left\langle\mathbf{J}\circ\phi_{\mathcal{A}},\tilde{\mathcal{A}}\right\rangle$
is $G$-invariant. It follows that the pulled back form,
$\iota_{(K)}^{\ast}\phi_{\mathcal{A}}^{\ast}\omega_{M}$ to the
$G$-manifold $\tilde{M}\times\tilde{\Orbit}_{(K)}$, is also
$G$-invariant. In fact, each term in
$\phi_{\mathcal{A}}^{\ast}\omega_M$ is independently invariant. To
check this, let
$\lambda:=((\alpha_{[m]},m),\nu_m)\in\tilde{M}\times\tilde{\Orbit}_{(K)}$,
(so that $\nu_{m}\in \mathfrak{h}_{(K)}^\circ$), and $v_{\lambda}\in
T_{\lambda}(\tilde{M}\times\tilde{\Orbit}_{(K)})$. Then,
\begin{align*}
\left\langle \mathbf{J}\circ\phi_{\A},\tA\right\rangle(g\cdot v_{\lambda})&=\left\langle \pi_{\nu^{\ast}}(g\cdot\lambda),\tA(g\cdot v_{\lambda})\right\rangle=\left\langle g\cdot \nu_m,\A(g\cdot T_{\lambda}(\tilde{\tau}\circ\pi_{\tilde{M}})(v_{\lambda}))\right\rangle\\
&=\left\langle g\cdot \nu_m,g\cdot\A (T_{\lambda}(\tilde{\tau}\circ\pi_{\tilde{M}})(v_{\lambda}))\right\rangle\\
&=\left\langle \nu_m,\A (T_{\lambda}(\tilde{\tau}\circ\pi_{\tilde{M}})(v_{\lambda}))\right\rangle=\left\langle \mathbf{J}\circ\phi_{\A},\tA\right\rangle(v_{\lambda}),
\end{align*}
from which it follows that, by infinitesimal invariance,
\begin{equation}
\label{lie derivative}
0=\mathbf{L}_{{\xi}_{\tilde{M}\times\tilde{\Orbit}_{(K)}}}\left\langle \mathbf{J}_{\tilde{\Orbit}_{(K)}},\tA\right\rangle
=\iota_{{\xi}_{\tilde{M}\times\tilde{\Orbit}_{(K)}}}\mathbf{d}\left\langle \mathbf{J}_{\tilde{\Orbit}_{(K)}},\tA\right\rangle
+\mathbf{d}\iota_{{\xi}_{\tilde{M}\times\tilde{\Orbit}_{(K)}}}\left\langle\mathbf{J}_{\tilde{\Orbit}_{(K)}},\tA\right\rangle,
\end{equation}
where $\mathbf{L}_{X}$ denotes the Lie derivative, and equation
\eqref{lie derivative} follows from Cartan's magic formula. Now,
consider the form $\omega^{\Orbit,(K)}$. The sum of the first two
terms is simply $\iota_{(K)}^{\ast}\phi_{\mathcal{A}}^{\ast}\omega$
and each is $G$-invariant, and the first term,
$(\pi^{\sharp})^{\ast}\omega_{M/G}$ is basic. The third term of
$\omega^{\Orbit,(K)}$ is easily checked to also be $G$-invariant
using equivariance of $\mathbf{J}_{\tilde{\Orbit}_{(K)}}$ and
invariance of the orbit symplectic form $\omega_{\Orbit}^{+}$.
Therefore $\omega^{\Orbit,(K)}$ is $G$-invariant. To see that
$\omega^{\Orbit,(K)}$ annihilates vertical vectors we must show that
\begin{equation}
\label{vertical kill}
\iota_{\xi_{\tilde{M}\times\tilde{\Orbit}_{(K)}}}
\left(-\mathbf{d}\left\langle\mathbf{J}_{\tilde{\Orbit}_{(K)}},\tA\right\rangle-\mathbf{J}_{\tilde{\Orbit}_{(K)}}^{\ast}\omega_{\Orbit}^{+}\right)=0.
\end{equation}
From equation \eqref{lie derivative} we have
$$\iota_{\xi_{\tilde{M}\times\tilde{\Orbit}_{(K)}}}\left(-\mathbf{d}\left\langle\mathbf{J}_{\tilde{\Orbit}_{(K)}},\tA\right\rangle\right)=\mathbf{d}\iota_{\xi_{\tilde{M}\times\tilde{\Orbit}_{(K)}}}\left\langle\mathbf{J}_{\tilde{\Orbit}_{(K)}},\tA\right\rangle=\mathbf{d}\left\langle\mathbf{J}_{\tilde{\Orbit}_{(K)}},[\xi]_\nu\right\rangle,
$$
where $[\xi]_\nu$ is the section of $\nu$ given by
$[\xi]_\nu(m)=[\xi]_m\in\g/{\g_m}$. We have,
\begin{equation*}
\xi_{\tilde{M}\times\tilde{\Orbit}_{(K)}}(\lambda)=\ddt ((\alpha_{[m]},\exp (t\xi)\cdot m),\operatorname{Ad}^{\ast}_{\exp (-t\xi)}\nu_m)
\end{equation*}
and therefore,
\begin{equation*}
T_{\lambda}\mathbf{J}_{\tilde{\Orbit}_{(K)}}\cdot \xi_{\tilde{M}\times\tilde{\Orbit}_{(K)}}(\lambda)=\ddt \operatorname{Ad}^{\ast}_{\exp (-t\xi)}\nu_m=-\operatorname{ad}^{\ast}_{\xi}\nu_m.
\end{equation*}
Furthermore, $T_{\lambda}\mathbf{J}_{\tilde{\Orbit}_{(K)}}$ applied to any tangent vector $v_{\lambda}\in T_{\lambda}(\tilde{M}\times\tilde{\Orbit}_{(K)})$ must be of the form
$-\operatorname{ad}^{\ast}_{\eta}\nu_m$ since the bundle $\tilde{\Orbit}_{(K)}=G\cdot (M_{H}\times\Orbit\cap \mathfrak{h}^\circ_{(K)})$ trivializes and saturates over $M_H$ and therefore the most general curve passing through $\lambda$ has the form
$t\mapsto ((\alpha_{[m]}(t),m(t)),\nu(t))$ where $\nu(t)=g(t)\cdot \nu_m$ for some curve $g(t)\in G$ through the identity. Therefore,
\begin{align*}
T_{\lambda}\mathbf{J}_{\tilde{\Orbit}_{(K)}}(v_{\lambda})&=\ddt \mathbf{J}_{\tilde{\Orbit}_{(K)}}((\alpha_{[m]}(t),m(t)),g(t)\cdot \nu_{m})\\
&=\ddt g(t)\cdot \nu_m=-\operatorname{ad}^{\ast}_{\eta}\nu_m
\end{align*}
where $\eta:=\ddt g(t)$.
Using this, we have
\begin{align*}
\iota_{\xi_{\tilde{M}\times\tilde{\Orbit}_{(K)}}}(\mathbf{J}_{\tilde{\Orbit}_{(K)}}^{\ast}\omega_\Orbit^{+})(v_{\lambda})&=
(\mathbf{J}_{\tilde{\Orbit}_{(K)}}^{\ast}\omega_\Orbit^{+})(\xi_{\tilde{M}\times\tilde{\Orbit}_{(K)}}(\lambda),v_{\lambda})\\
&=\omega_{\Orbit}^{+}(-\operatorname{ad}^{\ast}_{\xi}\nu_m,-\operatorname{ad}^{\ast}_{\eta}\nu_m)\\
&=\left\langle\nu_m,[\xi,\eta]\right\rangle.
\end{align*}
On the other hand,
\begin{align*}
\mathbf{d}\left\langle \mathbf{J}_{\tilde{\Orbit}_{(K)}},[\xi]\right\rangle(v_{\lambda})&=
\left\langle T_{\lambda}\mathbf{J}_{\tilde{\Orbit}_{(K)}}( v_\lambda),[\xi]\right\rangle\\
&=\left\langle-\operatorname{ad}^{\ast}_{\eta}\nu_m,\xi\right\rangle=\left\langle\nu_m,[\xi,\eta]\right\rangle
\end{align*}

so that equation \eqref{vertical kill} is satisfied. Consequently,
$\omega^{\Orbit,(K)}$ drops to a unique form
$\omega^{\Orbit,(K)}_\mathrm{min}$ on
$\tilde{M}\times_{G}\tilde{\Orbit}_{(K)}$, and this form satisfies
equation \eqref{reduced_form} since it is  the unique orbit reduced
form by construction.
\end{proof}

\begin{remark}
The above theorem generalizes the minimal coupling constructions for
the regular case given in Section 2 as follows. The form
$(\pi^{\sharp})^{\ast}\omega_{M/G}-\mathbf{d}\left\langle\mathbf{J}_{\tilde{\mathcal{O}}_{(K)}},\tilde{\mathcal{A}}\right\rangle$
is the singular generalization of $\tilde{\omega}_{\mathcal{O}}^{-}$
given in equation \eqref{tilde omega minus}. The form
$\omega^{\Orbit,(K)}_\mathrm{min}$ is the singular generalization of
the minimal coupling form $\omega_{\mathrm{min}}^{\Orbit}$ in
equation \eqref{definition minimal coupling form}.
\end{remark}

We then have the following result on the reduced symplectic form
$\omega^{\Orbit,(K)}_\mathrm{min}$ determined on each
$\tilde{M}\times_{G} \tilde{\Orbit}_{(K)}$ coupling the canonical
symplectic structure with the homenegous Kostant-Kirillov form on
the orbit fibers via the reduced curvature $\mathcal{B}$
 of the singular connection.

The involved mappings are summarized in the following diagram

\[
\xymatrix{
\tilde{M}\times\tilde{\Orbit}_{(K)}\ar[r]^{\pi_{\tilde M}}\ar[dd]_{\pi_{(K)}}&\tilde{M}\ar[r]^{\tilde{\tau}}\ar[dd]_{\pi^{\sharp}}&M\ar[dd]^{\pi}\\
\\
\tilde{M}\times_{G}\tilde{\Orbit}_{(K)}\ar[r]_{p}&T^{\ast}(M/G)\ar[r]_{\tau_{M/G}}&M/G.}
\]

We will proceed as before to write $\pi^\sharp=\pi^\sharp\circ
\pi_{\tilde M}$ in order to economize notation.

\begin{theorem}
\label{minimal coupling theorem} The reduced minimal coupling form
$\omega^{\Orbit,(K)}_\mathrm{min}$ on
$\tilde{M}\times_G\tilde{\mathcal{O}}_{(K)}$ can  be expressed in
terms of the reduced curvature, $\mathcal{B}$, of $\mathcal{A}$ as
follows.
\begin{equation*}
\omega^{\Orbit,(K)}_\mathrm{min}=(\pi^{\sharp})^{\ast}\omega_{M/G}-\left\langle
\tilde{\Phi},(\tau_{M/G}\circ
p)^{\ast}\mathcal{B}\right\rangle+\tilde{\omega}_{(K)}
\end{equation*}
where $p:\tilde{M}\times_{G}\tilde{\Orbit}:\rightarrow T^{\ast}(M/G)$ is the submersion given by $p([(\alpha_{[m]},m),\nu_m])= \alpha_{[m]}$ and
$\tilde{\Phi}([(\alpha_{[m]},m),\nu_m]):=[m,\nu_m]\in \tilde{\Orbit}_{(K)}$.

The two-form $\tilde{\omega}_{(K)}$ is
equivalent at each point to the homogeneous reduced symplectic form defined in \eqref{stratifiedKK}.

\end{theorem}
\begin{proof}

First, we observe that, as in the free case, using the connection $\mathcal{A}$ we can induce from the two-form $\mathbf{J}_{\tilde{\Orbit}_{(K)}}^{\ast}\omega_{\Orbit}^+$ on $\tilde{M}\times \tilde{\Orbit}_{(K)}$,  a form $\tilde{\omega}_{(K)}$ on
the quotient space $\tilde{M}\times_{G}\tilde{O}_{(K)}$. To construct this form, we need to first obtain a splitting of the tangent bundle of $\tilde{M}\times_{G}\tilde{\Orbit}_{(K)}$. Note that the vertical fibers of $p$ are given by
$$p^{-1}(\alpha_{[m]})\simeq (\Orbit_{\mu}\cap \mathfrak{h}^\circ_{(K)})/H$$ where $G_{m}=H$. These fibers carry a symplectic structure as we have determined in \eqref{stratifiedKK}. The tangent spaces to these fibers determine the vertical subbundle $V$ of $T(\tilde{M}\times_{G}\tilde{\Orbit}_{(K)})$: for each $[\lambda]=[(\alpha_{[m]},m),\nu_m]\in \tilde{M}\times_{G}\tilde{\Orbit}_{(K)}$,  $V_{[\lambda]}:=\mathrm{ker}\,T_{[\lambda]}p$. The connection, $\mathcal{A}$, determines a complement as follows. Given a vector $v_{\alpha_{[m]}}\in T_{\alpha_{[m]}}T^{\ast}(M/G)$ tangent to some curve $\alpha_{[m]}(t)$, denote the projected curve to $M/G$ by $[m](t):=\tau_{M/G}\circ\alpha_{[m]}(t)$. Then, letting $m(t)$ denote the horizontal lift of $[m](t)$ to $M$ through the point $m$ notice that, by construction, the curve $(\alpha_{[m]}(t),m(t))\in \tilde{M}$ and, furthermore, the curve is not contained in the $G$-orbit through $(\alpha_{[m]},m)$. Finally the tangent vector to the curve $((\alpha_{[m]}(t),m(t)),\nu_m)\in \tilde{M}\times\tilde{\Orbit}$ through $\lambda$, $(v_{(\alpha_{[m]},m)},0)$, is not contained in the kernel of $T_{[\lambda]}p$ and therefore we have constructed an injective map $\Phi_{[\lambda]}:T_{p([\lambda])}T^{\ast}(M/G)\rightarrow T_{[\lambda]}(\tilde{M}\times_{G}\tilde{\Orbit})$. The image of $\Phi_{[\lambda]}$ then has dimension complementary to $\mathrm{ker}\,T_{[\lambda]}p$. Now ranging over $[\lambda]$, this defines the horizontal distribution $H$ in $T(\tilde{M}\times_{G}\tilde{\Orbit})$ with $H_{[\lambda]}:=\mathrm{Im}\,\Phi_{[\lambda]}$. We then have a projection $P_{[\lambda]}$ onto the vertical space $V_{[\lambda]}$ corresponding to this splitting. The form $\tilde{\omega}_{(K)}$ is then defined on
$\tilde{M}\times_{G}\tilde{\Orbit}_{(K)}$ by
\begin{equation*}
%\label{dropped orbit form}
\tilde{\omega}_{(K)}(v_{[\lambda]},w_{[\lambda]}):=-(\mathbf{J}_{\tilde{\Orbit}_{(K)}}^{\ast}\omega_{\Orbit}^{+})((0,v),(0,w))
\end{equation*}
where $v \in T_{\nu_m}\tilde{\Orbit}_{(K)}$ satisfies
$T_{\lambda}\pi_{(K)}(0,v)=P_{[\lambda]}(v_{[\lambda]})$ and
analogously for $w$. This is well defined due to the $G$-invariance
of the form
$\mathbf{J}_{\tilde{\Orbit}_{(K)}}^{\ast}\omega_{\Orbit}^{+}$. Next,
consider the basic form (from equation \eqref{vertical kill}) on
$\tilde{M}\times \tilde{\Orbit}_{(K)}$,
$-\mathbf{d}\left\langle\mathbf{J}_{\tilde{\Orbit}_{(K)}},\tA\right\rangle-\mathbf{J}_{\tilde{\Orbit}_{(K)}}^{\ast}\omega_{\Orbit}^{+}$.
We need to show that
\begin{equation}
\label{curvature equality}
\pi_{(K)}^{\ast}\left(\left\langle \tilde{\Phi},(\tau_{M/G}\circ p)^{\ast}\mathcal{B}\right\rangle-\tilde{\omega}_{(K)}\right)=\mathbf{d}\left\langle\mathbf{J}_{\tilde{\Orbit}_{(K)}},\tA\right\rangle+\mathbf{J}_{\tilde{\Orbit}_{(K)}}^{\ast}\omega_{\Orbit}^{+}.
\end{equation}
Fix a point $\lambda:=((\alpha_{[m]},m),\nu_m)\in
\tilde{M}\times\tilde{\Orbit}_{(K)}$. We consider three types  of
tangent vectors at this point: $(h_{(\alpha_{[m]},m)},0)$,
$\xi_{\tilde{M}\times\tilde{\Orbit}_{(K)}}(\lambda)$, and $(0,u)$
where $h_{(\alpha_{[m]},m)}$ is a horizontal vector at
$T_{(\alpha_{[m]},m)}\tilde{M}$ relative to $\tilde A$, $\xi\in\g$
 and $u\in
T_{\nu_m}\tilde{\Orbit}_{(K)}$. We prove equation \eqref{curvature
equality} by checking equality on all six pairs of these types of
tangent vectors.  Note that on vectors of the form
$\xi_{\tilde{M}\times\tilde{\Orbit}_{(K)}}(\lambda)$ both sides of
equation \eqref{curvature equality} are zero: trivially for the left
hand side, and for the right hand side, by equation \eqref{vertical
kill}. Therefore it will suffice to verify equation \eqref{curvature
equality} on the three types of pairs generated by types
$(h_{(\alpha_{m},m)},0)$ and $(0,u)$, which we call horizontal and
momentum vectors respectively.
\paragraph {\bf{Horizontal, Horizontal}} Consider the pair $(\tilde{u}_1,0),(\tilde{u}_2,0)\in T_{\lambda}(\tilde{M}\times\tilde{\Orbit}_{(K)})$ where
$\tilde{u}_i$ is a horizontal vector on $\tilde{M}$ at the point
$(\alpha_{[m]},m)$ and
$u_i:=T_{(\alpha_{[m]},m)}\tilde{\tau}(\tilde{u_i})\in T_mM$. We use
the notational convention that an upper case letter is a vector
field that extends the lower case tangent vector. Extend these
vectors to vector fields $(\tilde{U}_i,0)$ where $\tilde{U}_i$ is a
horizontal vector field on $\tilde{M}$, so that the extended vector
fields are given by
$(\tilde{U}_i,0)((\alpha_{[m]},m),\nu_m)=(\tilde{U}_i(\alpha_{[m]},m),0)$.
Denoting by $U_i$ the uniquely defined $\tilde{\tau}$-related
horizontal  vector fields on $M$ given by $T\tilde{\tau}\circ
\tilde{U}_i={U}_i\circ\tilde{\tau}$, we then have,
\begin{align*}
\mathbf{d}\left\langle\mathbf{J}_{\tilde{\Orbit}_{(K)}},\tilde{\mathcal{A}}\right\rangle&((\tilde{u}_1,0),(\tilde{u}_2,0))=\\
&=(\tilde{U}_1,0)\left(\left\langle\mathbf{J}_{\tilde{\Orbit}_{(K)}},\tilde{\mathcal{A}}\right\rangle(\tilde{U}_2,0)\right)(\lambda)-
(\tilde{U}_2,0)\left(\left\langle\mathbf{J}_{\tilde{\Orbit}_{(K)}},\tilde{\mathcal{A}}\right\rangle(\tilde{U}_1,0)\right)(\lambda)\\&\qquad-
\left\langle\mathbf{J}_{\tilde{\Orbit}_{(K)}},\tilde{\mathcal{A}}\right\rangle\left([(\tilde{U}_1,0),(\tilde{U}_2,0)]\right)(\lambda)\\
&=-
\left\langle\mathbf{J}_{\tilde{\Orbit}_{(K)}},\tilde{\mathcal{A}}\right\rangle\left([(\tilde{U}_1,0),(\tilde{U}_2,0)]\right)(\lambda)\\
&=-\left\langle \nu_m,\tilde{\mathcal{A}}\left([\tilde{U}_1,\tilde{U}_2]\right)(\alpha_{[m]},m)\right\rangle\\
&=-\left\langle \nu_m, \mathcal{A}\left(T\tilde{\tau}[\tilde{U}_1,\tilde{U}_2]\right)(\alpha_{[m]},m)\right\rangle
=-\left\langle \nu_m, \mathcal{A}\left([U_1,U_2]\right)(m)\right\rangle\\
&=\left\langle \nu_m, B(u_1,u_2)\right\rangle,
\end{align*}

where we have used the fact that $T\tilde{\tau}\circ \tilde{U}_i=U_i\circ\tilde{\tau}$, so that
$T\tilde{\tau}\circ[\tilde{U}_1,\tilde{U}_2]=[U_1,U_2]\circ\tilde{\tau}$, and also the definition of the curvature
of the singular connection in the final equality.

We now compute the left hand side of equation \eqref{curvature equality} on the horizontal vectors,
\begin{align*}
\pi_{(K)}^{\ast}&\left(\left\langle\tilde{\Phi},(\tau_{M/G}\circ
p)^{\ast}\mathcal{B}\right\rangle\right)((\tilde{u}_1,0),
(\tilde{u}_2,0))=\\ &=\left\langle[m,\nu_m],(\tau_{M/G}\circ p)^{\ast}\mathcal{B}(T_{\lambda}\pi_{(K)}(\tilde{u}_1,0),T_{\lambda}\pi_{(K)}(\tilde{u}_2,0))\right\rangle\\
&=\left\langle [m,\nu_m],\mathcal{B}(T_m\pi\circ T_{(\alpha_{[m]},m)}\tilde{\tau}\circ T_{\lambda}\pi_{\tilde M}(\tilde{u}_1,0),
(T_m\pi\circ T_{(\alpha_{[m]},m)}\tilde{\tau}\circ T_{\lambda}\pi_{\tilde M}(\tilde{u}_2,0))\right\rangle\\
&=\left\langle \nu_m,\mathcal{B}(T_m\pi\circ T_{(\alpha_{[m]},m)}\tilde{\tau}(\tilde{u}_1),T_m\pi\circ T_{(\alpha_{[m]},m)}\tilde{\tau}(\tilde{u}_2))\right\rangle\\
&=\left\langle \nu_m,\mathcal{B}(T_m\pi (u_1),T_m\pi (u_2))\right\rangle=\left\langle \nu_m,B(u_1,u_2)\right\rangle,
\end{align*}
agreeing with the right hand side.

\paragraph{\bf{Momentum, Momentum}}
By definition of the form $\tilde{\omega}_{(K)}$ we have, for $(0,v),(0,w)\in T_{\lambda}(\tilde{M}\times\tilde{\Orbit}_{(K)})$,
\begin{align*}
\pi_{(K)}^{\ast}\left(\left\langle \tilde{\Phi},(\tau_{M/G}\circ p)^{\ast}\mathcal{B}\right\rangle-\tilde{\omega}_{(K)}\right)((0,v),(0,w))&=-
\pi_{(K)}^{\ast}\tilde{\omega}_{(K)}((0,v),(0,w))\\ &=\mathbf{J}_{\tilde{\Orbit}_{(K)}}^{\ast}\omega_{\Orbit}^{+}((0,v),(0,w))\\
&=\left(\mathbf{d}\left\langle\mathbf{J}_{\tilde{\Orbit}_{(K)}},\tA\right\rangle+\mathbf{J}_{\tilde{\Orbit}_{(K)}}^{\ast}\omega_{\Orbit}^{+}\right)((0,v),(0,w)).
\end{align*}
The first equality holds since $T_{\lambda}\pi_{(K)}(0,v)$ is in the kernel
of $T_{[\lambda]}p$.
The second equality holds by definition of the form $\tilde{\omega}_{(K)}$ and last follows by extending $(0,v)$ and $(0,w)$ to vector fields $(0,V),(0,W)$ where $V$ and $W$ are vector fields on $\tilde{\Orbit}_{(K)}$ and then using the fact that the bracket of these fields is just $(0,[V,W])$.

\paragraph{\bf{Horizontal, Momentum}}
On a pair of mixed vectors, both sides of equation \eqref{curvature
equality} vanish since horizontal vector fields commute with
momentum vector fields. We have now proven equation \eqref{curvature
equality}, and therefore by equation \eqref{the basic form}, the
theorem follows.

\end{proof}

\section{The Poisson stratification}
In this section we compute the Poisson stratification of the
reduced Poisson Sternberg space $S=\tilde{M}\times_G\nu^{\ast}$, and we write down the reduced gauge bracket on each of the
Poisson strata.

Following the setup of Section 5, the $G$-equivariant symplectomorphim $\phi_\mathcal{A}$ descends to a stratified isomorphism $\tilde{\phi}_\mathcal{A}:S\rightarrow (T^*M)/G$. Since $S$ is a quotient of a smooth manifold by a proper group action, it has a natural stratification with strata $S^{(K)}=(\tilde{M}\times\nu^*)_{(K)}/G$. Since $\tilde{M}$ is a single orbit type manifold, the orbit types of the pre-quotiented space correspond to the orbit types of the total space of the bundle $\nu^*$, which were shown, in the proof of Theorem \ref{lattice theorem}, to be in one-to-one correspondence with the isotropy lattice for the action $H\times \h^\circ\rightarrow \h^\circ$. Therefore the strata of $S$ are given by
$$S^{(K)}=\tilde{M}\times_G\nu^*_{(K)}.$$

Moreover, since $\tilde{M}\times \nu^*$ is a symplectic manifold
with  symplectic structure $\phi_{\mathcal{A}}^{\ast}\omega_M$,it
follows from the general theory that the orbit type strata of $S$
are actually Poisson, the Poisson structure coming from singular
reduction. The symplectic leaves of these Poisson structures are
exactly the manifolds $\tilde{M}\times_G\tilde{\Orbit}_{(K)}$
equipped with the minimal coupling forms
$\omega^{\Orbit,(K)}_\mathrm{min}$ of Theorem \ref{minimal coupling
theorem}. Instead of using the theory of singular reduction to
compute the reduced Poisson brackets on the strata $S^{(K)}$, in
Theorem \ref{gauge poisson theorem} we will postulate these brackets
and then verify that they actually produce the minimal symplectic
foliation of Theorem \ref{minimal coupling theorem}. By the
uniqueness of the symplectic foliation of a Poisson manifold it
follows that the postulated brackets are actually the reduced gauge
Poisson brackets.

\begin{theorem}
\label{gauge poisson theorem} Let $s=[(\alpha_{[x]},x),\mu_x]\in
S^{(K)}$, where without loss of generality we have chosen $G_x=H$
and $H_{\mu_x}=K$. Let $f,g\in C^\infty(S^{(K)})$. Then the reduced
Poisson structure on $S^{(K)}$ is given by
\begin{equation*}\begin{array}{lll}\{f,g\}_{S^{(K)}}(s) & = & \omega_{M/G}
(\alpha_{[x]})\left(\mathbf{d}_{\tilde{\mathcal{A}}}^Sf(s)^\sharp,\mathbf{d}_{\tilde{\mathcal{A}}}^Sg(s)^\sharp\right)
+\left\langle
[\mu_x],\tilde{\mathcal{B}}(\alpha_{[x]})(\mathbf{d}_{\tilde{\mathcal{A}}}^Sf(s)^\sharp,
\mathbf{d}_{\tilde{\mathcal{A}}}^Sg(s)^\sharp\right\rangle\vspace{8mm}\\
 & - & \left\langle\mu_x,\left[\frac{\delta
\overline{f}^\circ}{\delta\mu },\frac{\overline{g}^\circ}{\delta\mu
}\right]\right\rangle , \end{array}\end{equation*}
where the sharp operator in the first term is with respect to $\omega_{M/G}$ and the covariant derivative $\mathbf{d}_{\tilde{\mathcal{A}}}^S$ is computed using the singular connection as in equation \eqref{covariant derivative} for the regular case. In the second term, $\tilde{\mathcal{B}}=\tau_{M/G}^{\ast}\mathcal{B}$ where $\mathcal{B}$ is the reduced curvature of the singular connection, $\mathcal{A}$. For the third term, $\overline{f}$ is the restriction of $f$ to the fiber of $\nu^*_{(K)}/G$ at $[x]$, which is isomorphic to $\h^\circ_{(K)}/H$, and $\overline{f}^\circ$ is a $H$-invariant extension to $\g^*$ of the lift of $\overline{f}$ to $\h^\circ_{(K)}$.
Note that the last
term is simply a $(-)$ homogeneous Lie-Poisson bracket
$$\{\overline{f},\overline{g}\}_{(K)}([\mu_x])$$ on the fiber of $\nu^*_{(K)}/G$ at $[x]$ as introduced in Section 3, in equation \eqref{homogeneous bracket}.
\end{theorem}
\begin{proof}
Since $S^{(K)}$ is Poisson, the theorem is proved if for any pair
$f,g\in C^\infty(S^{(K)})$ and $s\in
\tilde{M}\times_G\tilde{\Orbit}_{(K)}\subset S^{(K)}$,
$\{f,g\}_{S^{(K)}}(s)=\omega^{\Orbit,(K)}_\mathrm{min}(X_f(s),X_g(s))$,
where $X_f,X_g$ are the Hamiltonian vector fields associated to the
restrictions of $f$ and $g$ to
$\tilde{M}\times_G\tilde{\Orbit}_{(K)}$. Let $U=O\times
\mathbb{R}^n$ be a trivializing neighborhood of $T^*(M/G)$  and
shrink $O$ and $\mathbb{R}^n$ if necessary so that $S^{(K)}$ is
trivialized over $U$ like $S_U^{(K)}=U\times
\mathfrak{h}^\circ_{(K)}/H$. If $x^i,\,p_i,\,i=1,\ldots,n$ are
bundle coordinates on $U$ we will consider the family of functions
$x^i,\,p_i$ and $f\in C^\infty(\mathfrak{h}^\circ_{(K)}/H)$, whose
differentials span the cotangent bundle of $S^{(K)}$ at any point of
$S_{U}^{(K)}$.

Let $(x^1,\ldots,x^n,g^1,\ldots,g^k)$ be local coordinates on $M$
over $O$ and $A_i^l,\,i=1,\ldots,n+k,\,l=1,\ldots,k$ the components
of the connection $\mathcal{A}$. The local coordinates $\{g^l\}$ on
$G$ are chosen in a way that $A(\partial_{g^l})=\xi_l$, where
$\{\xi_1,\ldots,\xi_k\}$ is a basis for $\g$ for which, for $r<k$, $\{\xi_1,\ldots,\xi_r\}$ is a basis for $\g_x$.
Then the horizontal
lift of a local vector field $\partial_{x^i}$ on $O$ is
$\partial_{x^i}-A_i^l\partial_{g^l}.$

Therefore, we have
$\mathrm{hor}_s(\partial_{x^{i}})=\partial_{x^{i}}$ and
$\mathrm{hor}_s(\partial_{p_{i}})=\partial_{p_{i}}$. Consequently we
obtain
$$\begin{array}{lll}\mathbf{d}_{\tilde{\mathcal{A}}}^Sx^i(s) & = &
dx^i\\
\mathbf{d}_{\tilde{\mathcal{A}}}^Sp_i(s) & = &
dp_i\\
\mathbf{d}_{\tilde{\mathcal{A}}}^Sf(s) & = & 0\end{array}$$ Since in
this trivialization $\omega_{M/G}$ is given by $dx^i\wedge dp_i$ it
follows that
$$\begin{array}{lll}\mathbf{d}_{\tilde{\mathcal{A}}}^Sx^i(s)^\sharp & = &
-\partial_{p_i}\\
\mathbf{d}_{\tilde{\mathcal{A}}}^Sp_i(s)^\sharp & = &
\partial_{x^i}\\
\mathbf{d}_{\tilde{\mathcal{A}}}^Sf(s)^\sharp & = & 0.\end{array}$$

Next, linear coordinates for $\g^*$ with respect to the dual basis
$\{\xi_1,\ldots,\xi_k\}$ are given by $\{\mu_1,\ldots,\mu_k\}$. Let
$B_{ij}^\alpha$, $i,j=1,\ldots,n,\,\alpha=r+1,\ldots,k$

be the local expression for the components of $\mathcal{B}$, the
reduced curvature of $\mathcal{A}$. Then the local expressions for
the bracket in the statement of the theorem are given by

\begin{equation}\label{gauglocal}\begin{array}{lll}
\{x^i,p_j\}(s) & = & \delta^i_j\\
\{p_i,p_j\}(s) & = & \mu_\alpha B_{ij}^\alpha ,\, \alpha=r+1,\ldots,k\\
\{x^i,x^j\}(s) & = & 0\\
\{x^i,f\}(s) & =& 0\\
\{p_i,f\}(s) & = & 0\\
\{f,g\}(s) & = & -\left\langle\mu,\left[\frac{\partial
f^\circ}{\partial\mu},\frac{\partial
g^\circ}{\partial\mu}\right]\right\rangle .
\end{array}\end{equation}
This is easily checked to define a Poisson tensor. The only point that requires a straightforward calculation is to check the Jacobi identity on a bracket of type $\{p_i,\{p_j,p_k\}\}$. But, since the singular curvature satisfies the Bianchi identity (see Remark \ref{ehresmann approach}), it follows that the two-form $\left\langle [\mu],\mathcal{B}\right\rangle$ is closed which implies Jacobi. Antisymmetry is implied by the antisymmetry of the reduced curvature form.

Note that the last expression is nothing but the homogeneous
Lie-Poisson bracket of Section 3, and that it restricts on each momentum fiber of
$\tilde{M}\times_G\tilde{\mathcal{O}}_{(K)}$ over $T^*(M/G)$ to the homogeneous reduced
symplectic form in \eqref{stratifiedKK}.

We now compute the Hamiltonian vector fields on a given typical
symplectic leaf $\tilde{M}\times_G\tilde{\mathcal{O}}_{(K)}$ with
respect to $\omega^{\Orbit,(K)}_\mathrm{min}$, as well as the
Poisson structure on these leaves induced by
$\omega^{\Orbit,(K)}_\mathrm{min}$.

It easily follows from Theorem \ref{minimal coupling theorem} that
in our local coordinates,
$$\omega^{\Orbit,(K)}_\mathrm{min}=dx^i\wedge dp_i-\mu_\alpha B^\alpha_{ij}\,dx^i\wedge
dx^j+\tilde\omega_{(K)},\quad i,j=1,\ldots,n,\,\alpha=r+1,\ldots
k.$$

From here, we  immediately obtain the associated Hamiltonian vector
fields
\begin{eqnarray*}
X_{x^i} & = & -\partial_{p_i}\\
X_{p_i} & = & \partial_{x^i} -\mu_\alpha B^\alpha_{ij}\partial_{p_j}
\end{eqnarray*}
for $i,j=1,\ldots,n,
\,\alpha=r+1,\ldots k.$ If $f\in C^\infty(\h^\circ_{(K)}/H)$, we
denote also by $f$ its restriction to $\mathcal{O}_{(K)}/H$, the
typical fiber of the fibration
$p:\tilde{M}\times_G\tilde{\mathcal{O}}_{(K)}\rightarrow T^*(M/G)$.
Then $X_f$ is defined by
$$\tilde\omega_{(K)}(X_f,\cdot )=df.$$

From these local expressions it is now clear that

$$\begin{array}{lll}
\omega^{\Orbit,(K)}_\mathrm{min}(X_{x^i},X_{x^j}) & = & 0\\
\omega^{\Orbit,(K)}_\mathrm{min}(X_{x^i},X_{p_j}) & = & \delta_i^j\\
\omega^{\Orbit,(K)}_\mathrm{min}(X_{p_i},X_{p_j}) & = & -\mu_\alpha B_{ji}^\alpha=\mu_\alpha B_{ij}^\alpha\\
\omega^{\Orbit,(K)}_\mathrm{min}(X_{x^i},X_f) & = & 0\\
\omega^{\Orbit,(K)}_\mathrm{min}(X_{p_i},X_f) & = & 0\\
\omega^{\Orbit,(K)}_\mathrm{min}(X_f,X_g) & = &
\tilde\omega_{(K)}(X_f,X_g)=\{f,g\}_{(K)\vert
\tilde{M}\times_G\tilde{\mathcal{O}}_{(K)}}
\end{array}$$

which agree with \eqref{gauglocal} by \eqref{homogeneous bracket}
and the discussion of Section 3.
\end{proof}

\end{document}